\documentclass[]{interact}

\usepackage{epstopdf}
\usepackage[caption=false]{subfig}

\usepackage[numbers,sort&compress]{natbib}
\bibpunct[, ]{[}{]}{,}{n}{,}{,}
\renewcommand\bibfont{\fontsize{10}{12}\selectfont}

\theoremstyle{plain}
\newtheorem{theorem}{Theorem}[section]

\newtheorem{corollary}[theorem]{Corollary}

\theoremstyle{definition}
\newtheorem{definition}[theorem]{Definition}
\newtheorem{example}[theorem]{Example}

\theoremstyle{remark}

\newcommand\mystyle{\everymath{\displaystyle}}
\mystyle
\usepackage{xcolor}

\begin{document}

\articletype{Research Article}

\title{Long-Lasting and Slowly Varying Transient Dynamics in Discrete-Time Systems}

\author{
\name{Anthony Pasion\textsuperscript{a, b}
and Felicia Maria Magpantay \textsuperscript{a}\thanks{CONTACT F.~M.~Magpantay Email: felicia.magpantay@queensu.ca}\thanks{This article has been accepted for publication in \emph{Journal of Difference Equations and Applications}, published by Taylor \& Francis. The final published version will be available through the journal.}}
\affil{\textsuperscript{a}Department of Mathematics and Statistics, Queen’s University, Kingston, Ontario, Canada; \textsuperscript{b}Department of Mathematics and Computer Science, University of the Philippines Baguio, Baguio City, Philippines}
}

\maketitle


\begin{abstract}
    Analysis of mathematical models in ecology and epidemiology often focuses on asymptotic dynamics, such as stable equilibria and periodic orbits. However, many systems exhibit long transient behaviors where certain aspects of the dynamics remain in a slowly evolving state for an extended period before undergoing rapid change. In this work, we analyze long-lasting and slowly varying transient dynamics in discrete-time systems based on extensions of previous theoretical frameworks developed for continuous-time systems. This involves clarifying the conditions under which we say an observable of the system exhibits prolonged transients, and deriving criteria for characterizing these dynamics. Our results show that specific points in the state space, analogous to previously defined transient centers in continuous-time systems, can generate and sustain long transients in discrete-time models. We demonstrate how these properties manifest in predator-prey models and epidemiological systems, particularly when populations or disease prevalence remain low for extended intervals before sudden shifts. 
\end{abstract}

\begin{amscode}
    39A05; 39A60; 37N25
\end{amscode}


\section{Introduction}

The dynamics of ecological and epidemiological systems are often described with a focus on their long-term behavior. However, recent advances have highlighted that many such systems exhibit transient behavior that can persist for ecologically relevant time spans before transitioning to a qualitatively different regime. These long transients can have significant implications for forecasting, management, and conceptual understanding of population persistence, extinction, and disease dynamics. As such, there is growing interest in developing a systematic mathematical framework to characterize and predict long-lasting transient dynamics. 

A foundational body of work led by Hastings et al. has emphasized the importance of transient dynamics in ecology \cite{hastings2004transients,hastings2018transient,morozov2020long,morozov2024long}. They highlighted that system trajectories may remain far from their attractors for a prolonged period, potentially leading to misleading conclusions if only asymptotic dynamics are considered. Long transients have been observed in models of predator-prey systems, population models with dispersal-driven dynamics, and models of epidemic outbreaks. These behaviors may arise from mechanisms such as crawl-by dynamics, ghost attractors, chaotic repellers, chaotic saddles or invariant manifolds \cite{hastings2004transients,hastings2018transient,morozov2020long,morozov2024long, morozov2023long, morozov2024regime}.

In an attempt to give these ideas a more technical mathematical framework, Liu and Magpantay (2022) introduced the notions of transient points, transient centers, and reachable transient centers for continuous-time systems governed by ordinary differential equations (ODEs)\cite{liu2022quantification}. These are points in state space around which an observable quantity evolves very slowly for a long time. These ideas were further developed in \cite{liu2023framework} and used to gain insight into transient phenomena in ecology such as the honeymoon period of a disease \cite{mclean1988measles,kharazian2020honeymoon} or the temporary collapse of the population in predator-prey interactions as seen in \cite{hastings2018transient}. 

The Liu et al. framework \cite{liu2022quantification, liu2023framework} was developed in the setting of model systems governed by ODEs. However, many biological systems are more naturally modeled using discrete-time dynamical systems governed by difference equations or maps. These include systems with discrete generations, seasonal forcing terms, or pulsed interventions. Discrete-time models are also known to exhibit rich and varied transient dynamics, which may not have clear analogues in continuous-time, especially when governed by discontinuous or piecewise-defined maps \cite{morozov2023long,morozov2024regime}.  

The purpose of this work is to extend the framework developed by Liu et al. \cite{liu2022quantification,liu2023framework} to systems governed by discrete-time maps. Here we define and characterize points in the state space that generate slowly varying and long-lasting transient dynamics in discrete-time systems. Our results show that, just as in continuous-time models, discrete-time systems can feature special points in the state space where the evolution of an observable of the system is nearly stationary for extended periods before undergoing sudden transitions. These theoretical results are applied to various ecological and epidemiological models. By grounding our results in a rigorous mathematical setting and linking them to some well-documented transient phenomena in the literature, we hope that this work will contribute to the ongoing development of a unified mathematical theory of long transients across modeling frameworks. 

The paper is divided into five (5) sections. In Section 2, we introduce precise mathematical definitions of concepts such as transient points and transient centers adapted specifically for discrete-time systems. We also establish foundational results on their properties and derive necessary conditions for their existence. In Section 3, we examine conditions under which fixed points can serve as transient centers and provide clear criteria for their identification based on system stability and observable properties. In Section 4, we apply the theoretical framework we have developed to specific ecological and epidemiological models. Transient phenomena such as prolonged population collapses and honeymoon periods observed in disease outbreaks are explicitly illustrated. We end the paper by summarizing its main  contributions and lay out our future directions that build on these results.

\section{Notation and Preliminary Results}

We denote the set of all positive integers by $\mathbb{Z}^+$, the set of all nonnegative integers by $\mathbb{Z}^+_0$, and the $n$-dimensional real and complex Euclidean space by $\mathbb{R}^n$ and $\mathbb{C}^n$, respectively. We use the notation $\|\cdot \|$ to denote the Euclidean norm on $\mathbb{R}^n$ or $\mathbb{C}^n$. For $r>0$ and $x \in \mathbb{R}^n$ (or $x \in \mathbb{C}^n$), we let $B_r(x)=\{x~:~\|x\|<r\}$. If $n=1$, we simply write the spaces $\mathbb{R}^1$ and $\mathbb{C}^1$ as $\mathbb{R}$ and $\mathbb{C}$, respectively and use the symbol $|\cdot |$ to denote the modulus of a point in $\mathbb{R}$ or $\mathbb{C}$. We write $\mathbb{R}^{n\times n}$ for the set of all $n\times n$ real matrices. For integers $t_1$ and $t_2$, we use the shorthand $t_1\mathbin{:}t_2=\{\,t\in\mathbb{Z}:\ t_1\le t\le t_2\,\}$ for brevity. Hence, the notation $t\in 0\mathbin{:}T$ means that $t=0,1,\dots,T$. Let $C^k(\mathbb{R}^n, \mathbb{R}^n)$ be the space of all functions $f:~\mathbb{R}^n \to \mathbb{R}^n$ for which all partial derivatives up to order $k$ exist, for $k\in\mathbb{Z}^+_0$, and are continuous. In particular, $C(\mathbb{R}^n, \mathbb{R}^n) = C^0(\mathbb{R}^n, \mathbb{R}^n)$ denotes the space of continuous functions. Lastly, for a map $f\in C^1(\mathbb{R}^n,\mathbb{R}^n)$ and a point $x\in\mathbb{R}^n$, we use the notation $Df(x)$ to denote the Jacobian matrix of $f$ at the point $x$. 

Let us now consider the following discrete-time system given by an autonomous recursion equation,
\begin{align}
\begin{cases}
	x(t+1) = f(x(t)) & \\
	x(0) = \xi, &
\end{cases}  \label{dtades}
\end{align}
where $f \in C(\mathbb{R}^n, \mathbb{R}^n)$ and $\xi \in \mathbb{R}^n$. Since we are dealing with discrete-time systems, we always assume that $t\in \mathbb{Z}^+_0$ unless otherwise specified. For each $\xi \in \mathbb{R}^n$, we denote by $f^t(\xi)$ the value at time $t$ of the solution to \eqref{dtades} starting at $\xi$. Then $f^t(\xi) = f(f^{t-1}(\xi))$ and $f^0(\xi) = \xi$. The continuity of $f$ is sufficient to guarantee the existence and uniqueness of solution for each initial state $\xi \in \mathbb{R}^n$ (see for instance Section 4.2 of \cite{elaydi2005introduction}). In this work, the terms trajectory and orbit are used interchangeably to refer to the sequence of states a system visits over time, starting from an initial point $\xi$, i.e. the set $\{f^t(\xi)\}_{t\in\mathbb{Z}^+_0}$. 

The following definitions of transient points and transient centers adapt the framework by Liu et al. \cite{liu2022quantification,liu2023framework} to discrete-time systems. 

\begin{definition}[Transient Points]
Let $\xi \in \mathbb{R}^n$, $v \in C(\mathbb{R}^n, \mathbb{R})$, $s > 0$ and $T \in \mathbb{Z}^+$. Define the difference operator of $v$ along the solutions to \eqref{dtades} as 
	\begin{align*}
	    \Delta v (x) = v(f(x)) - v(x).
	\end{align*}
	Now, define  
	\begin{align*}
	T_s(\xi) = \begin{cases}
		\inf \{ t \in \mathbb{Z}^+_0 ~:~ | \Delta v(f^t(\xi)) | > s \}, & \text{if the set is non-empty,} \\
		\infty, & \text{otherwise}.
	\end{cases}
	\end{align*}
We call $T_s(\xi)$ the $(v,s)$-transient time of $\xi$ and we say that $\xi$ is a $(v, s, T)$-transient point if $T < T_s(\xi) < \infty$.

\end{definition}

\begin{definition}[Transient Centers] \label{defn: transient center}
	Let $v \in C(\mathbb{R}^n, \mathbb{R})$. We call $x_* \in \mathbb{R}^n $ a $v$-transient center if there exists $S > 0$ such that for all $0 < s < S$ and all $T \in \mathbb{Z}^+$, there exists a $(v, s, T)$-transient point in every neighborhood of $x_*$.
\end{definition}

 \begin{figure}
    \centering
    \includegraphics[width=7cm]{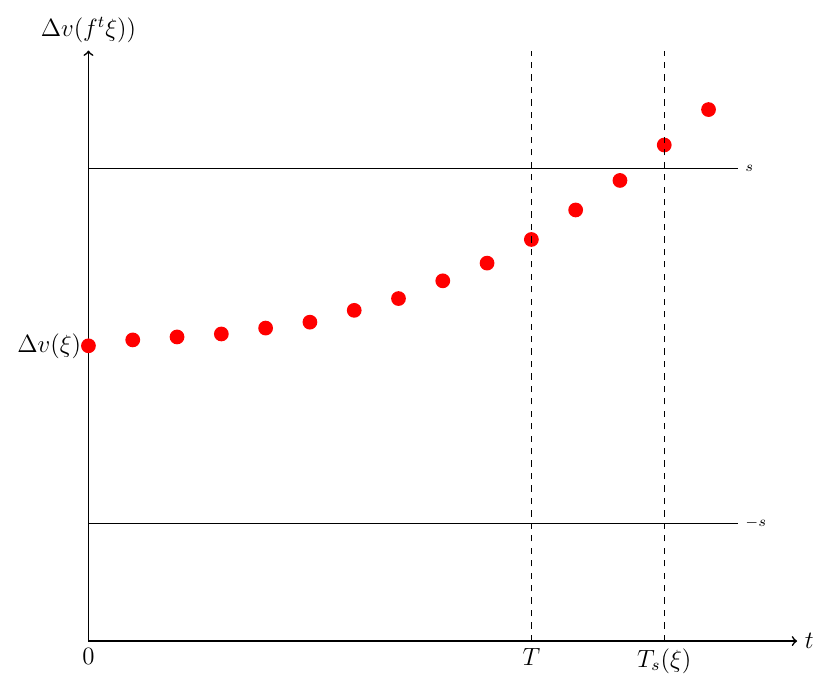}
    \caption{Illustration of the behaviour of $\Delta v(x(t))$ starting from a  $(v, s, T)$-transient point $\xi \in \mathbb{R}^n$.}
    \label{fig: transient}
\end{figure}

It is common to model systems wherein not all the states are observed or measured. We define an \textit{observable function} $v: \mathbb{R}^n \to \mathbb{R}$, which captures some quantity of interest evolving along system trajectories. 
Examples of such observables include the number of infected individuals in epidemic models, the density of prey or predator populations in ecological predator-prey models, total biomass in food-web models, or the fraction of vaccinated individuals in disease models with vaccination. Throughout the remainder of this paper, we assume that $v \in  C(\mathbb{R}^n, \mathbb{R})$ unless stated otherwise. We are interested in long and slow transient behaviours displayed by the observable $v$. A $(v,s,T)$-transient point is a state $\xi$ where the rate of change of the observable $v$ remains small for a long time. Specifically, $\Delta v$ stays below a threshold $s$ for at least time $T$. This means that the system exhibits slow evolution in terms of the observable $v$ before eventually experiencing a more significant rate of change. On the other hand, a transient center is a point in state space where arbitrarily slow transient behaviors can occur arbitrarily close to that point. Specifically, this means for any threshold $s$, no matter how small, there are states arbitrarily close to $x_*$ where $v$ changes slower than $s$ for longer than any given time interval $T$, no matter how long. Transient centers act as organizing structures in the dynamics, where the system's observable experiences prolonged phases of near-stationary behavior before transitioning to different regimes. Intuitively, if $x_*$ is a transient center, then the system exhibits long transient dynamics near $x_*$, meaning that trajectories of the observable originating close to $x_*$ remain in a slowly evolving phase for extended durations before eventually undergoing more noticeable changes. 

Several foundational results we establish in this section, including necessary conditions and invariance properties, are direct discrete-time analogues of the continuous-time results presented by Liu et al. in \cite{liu2023framework}. However, this paper also introduces new results that are tailored for discrete-time maps. These new results provide simplified and useful criteria for identifying transient centers, significantly facilitating their verification in discrete-time models. We now proceed to rigorously present these theoretical results.

\begin{theorem} \label{thm: zero}
 If $x_*$ is a $v$-transient center, then
 \begin{enumerate}
\item [(i)] $x_*$ is also an $(\alpha v + \beta)$-transient center for any $\alpha \neq0$ and $\beta \in \mathbb{R}$,
\item [(ii)] $\Delta v(x_*) = v(f(x_*)) - v(x_*) = 0$, and
\item [(iii)] $f^t(x_*)$ is also a $v$-transient center for all $t \in \mathbb{Z}^+_0$.
 \end{enumerate}
\end{theorem}

\begin{proof}
\textit{Proof of (i).} 
We use the notation $T^v_s(\xi)$ to denote the $(v,s)$-transient time of $\xi$.
Let $\tilde{v} = \alpha v +\beta$ where $\alpha \neq0$ and $\beta \in \mathbb{R}$. 
If $\xi \in \mathbb{R}^n$ is a $(v, s, T)$-transient point then $T <T^v_s(\xi) < \infty$. Clearly, $\xi$ is also a $(\tilde{v}, s/|\alpha|, T)$-transient point since its $(\tilde{v}, s/|\alpha|)$-transient time $T^{\tilde{v}}_{s/|\alpha|}(\xi) = T^v_s(\xi)$ from the fact that $\Delta \tilde{v} = \alpha \Delta v$. It is now straightforward from Definition \ref{defn: transient center} that $x_*$ is also a $\tilde{v}$-transient center. 

\textit{Proof of (ii).} 
Let $|\Delta v(x_*)| =| v(f(x_*)) - v(x_*) |= c$ and assume that $c\neq 0$. Since $v$ and $f$ are both continuous, the function $\Delta v = v(f(x))-v(x)$ is also continuous on $\mathbb{R}^n$. By the definition of continuity, there exists a neighborhood $U$ of $x_*$ such that for all $x \in U$, 
	\begin{align*}
		|\Delta v(x) - \Delta v(x_*)|  < \frac{c}{2}.
	\end{align*}
	Rearranging gives us $| \Delta v(x) | > \frac{c}{2}$ for all $x \in U$. 
Now, set $s = \frac{c}{2}$. By the definition of a transient center, there must exist a $(v, s, T)$-transient point arbitrarily close to $x_*$. However, in the neighborhood $U$, we have $|\Delta v(x)| > s$ or in other words, $T_s(x)=0$ which contradicts this requirement. Thus we must have $c = 0$. 

\textit{Proof of (iii).} We proceed by induction on $t$. By assumption, $x_*$ is a $v$-transient center, so the statement holds for $t = 0$ since $f^0(x_*) = x_*$. We now assume that for some $t \in \mathbb{Z}^+_0$, $f^t(x_*)$ is a $v$-transient center and show that  $f^{t+1}(x_*)$ is also a $v$-transient center. Let $C$ be an arbitrary neighborhood of $f^{t+1}(x_*)$. 
The continuity of $f$ implies the existence of a neighborhood $B$ of $f^t(x_*)$ such that $f(B) \subset C$. By the inductive hypothesis, there exists a $(v, s, T+1)$-transient point $x$ in $B$. By applying $f$, we see that $f(x) \in C$ and satisfies the condition of being a $(v, s, T)$-transient point. Since $C$ was an arbitrary neighborhood of $f^{t+1}(x_*)$, we conclude that $f^{t+1}(x_*)$ is a $v$-transient center. This completes the induction step. It follows that $f^t(x_*)$ is a $v$-transient center for all $t \in \mathbb{Z}^+_0$.
\end{proof}

The previous theorem provides a characterization of transient centers in discrete-time systems. Theorem \ref{thm: zero})(i) tells us that the property of being a transient center does not depend on the particular units or origin of the observable $v$, as long as you are only scaling or shifting $v$. For example, if $v$ represents a certain population component of a model system, then it does not matter whether the population is in hundreds, thousands or even subtracted from some baseline value. Theorem \ref{thm: zero})(ii) states that if $x_*$ is a transient center, then the observable function $v$ does not change at $x_*$ under iterations of the system, that is,  $v(f(x_*)) = v(x_*)$.  Note that the converse does not hold in general. This is easy to see by considering the case where $x_*$ is a fixed point of $f$, i.e., $f(x_*) = x_*$. In this case, we have $\Delta v(x_*) = v(f(x_*)) - v(x_*) = v(x_*) - v(x_*) = 0$. However, this does not imply that transient points accumulate around $x_*$ in the manner required by the definition of a transient center. In the next section, we shall develop rigorous criteria for identifying which fixed points act as transient centers. Finally, Theorem \ref{thm: zero})(iii) shows that if $x_*$ is a $v$-transient center, then every forward iterate $f^t(x_*)$ for all $t \in \mathbb{Z}^+_0$ remains a transient center. This means that once a point exhibits transient center behavior, this property persists along its forward trajectory under $f$. Combining this result with Theorem \ref{thm: zero})(ii), it follows that $v$ remains constant along the entire trajectory $\{ f^t(x_*) \}_{t \in \mathbb{Z}^+_0}$. Although transient centers are not necessarily fixed points of $f$, they behave similarly with respect to the observable $v$. This implies that transient centers form an invariant set under $f$ where $v$ does not change significantly over time. If $v$ represents our quantity of interest, then neighborhoods of transient centers are regions where this quantity can change very slowly over time.  
 
Let us now define the following set,
\begin{align}
	X^v =  \{ \xi \in \mathbb{R}^n~|~ \Delta v(f^t(\xi)) =0~\text{for all}~t \in \mathbb{Z}_0^+\} \label{X}
\end{align}
consisting of all points where the observable function $v$ remains constant under iteration. The set $X^v$ contains candidate $v$-transient centers. The following theorems provide useful criteria to determine which elements of the set $X^v$ defined in \eqref{X} are indeed transient centers.

\begin{theorem} \label{thm: transient center criterion}
	Let $x_* \in X^v$. If there exists $S > 0$ such that for all $s \in(0,S)$, in any neighborhood of $x_*$, there is an $x$ such that $T_s(x) < \infty$, then $x_*$ is a $v$-transient center.
\end{theorem}


\begin{proof}
Let $s \in (0, S)$ and $T \in \mathbb{Z}^+$. Let $U$ be a neighborhood of $x_*$. We will show that $U$ contains a $(v, s, T)$-transient point.  Since $x_* \in X^v$, we have $\Delta v(f^t(x_*)) = 0$ for all $t \in \mathbb{Z}^+_0$. For each $t\in 0\mathbin{:}T$, from the continuity of $\Delta v$ and $f^t$,  we can choose $r_t>0$ such that for all $x\in B_{r_t}(x_*)$ we have $|\Delta v (f^t(x))| < s$. Here, we also pick $r_0$ small enough so that $B_{r_0}(x_*) \subset U$.
Now set $r=\min_{t\in 0\mathbin{:}T}\{r_t\}$ and $W=B_r(x_*)$.
Clearly $W \subset B_{r_0}(x_*)  \subset U$. By construction, $|\Delta v(f^t(x))| < s$ for all $x \in W$ and $t \in 0\mathbin{:}T$. By assumption, there exists a point $x \in W$ such that $T_s(x) < \infty$. Since $x \in W$, we know that $|\Delta v(f^t(x))| < s$ for all $t \in 0\mathbin{:}T$, ensuring that $T_s(x) \geq T$. Thus, $x$ satisfies the conditions for being a $(v, s, T)$-transient point. It follows that $x_*$ is a $v$-transient center.
\end{proof}

\begin{theorem}\label{thm: finite transient time}
	A point $x_* \in X^v$ is a $v$-transient center if and only if there exists $S_* > 0$ such that in any neighborhood of $x_*$, there is an $x$ such that $T_{S_*}(x) < \infty$.
\end{theorem}

\begin{proof}
Assume that $x_*$ is a $v$-transient center. Then, by definition, there exists some $S > 0$ such that for all $s\in(0,S)$ and all $T \in \mathbb{Z}^+$, any neighborhood $U$ of $x_*$ contains a $(v, s, T)$-transient point. Choosing $S_* = S/2$, we see that there must exist an $x \in U$ which is a $(v, S_*, T)$-transient point for any $T \in \mathbb{Z}^+$ since $x_*$ is a $v$-transient center. This means $T_{S_*}(x) < \infty$, which proves the forward implication. Conversely, suppose that there exists $S_* > 0$ such that for every neighborhood of $x_*$, there is an $x$ with $T_{S_*}(x) < \infty$. For any $s\in(0,S_*)$, clearly $T_s(x) \leq T_{S_*}(x) < \infty$.
From the previous theorem we conclude that $x_*$ is a $v$-transient center. This completes the proof.
\end{proof}

Theorem \ref{thm: finite transient time} reduces the problem of verifying transient centers from a universal condition (e.g., for all $0 < s <S$) to an existential one (e.g., there exists $S_*$). Additionally, we can also restate Theorem \ref{thm: finite transient time} in an equivalent double supremum form since the condition that there exists $S_* > 0$ so that every neighborhood of $x_*$ contains some $x$ with $T_{S_*}(x) < \infty$ is equivalent to the existence of $S_* >0$ such that for every $r>0$,
\begin{align*}
    \sup_{x \in B_r(x_*)}~ \sup_{t \in \mathbb{Z}^+_0} \left| \Delta v(f^t(x)) \right| \geq S_*.
\end{align*}
In this form, no matter how small the ball $B_r(x_*)$, the maximum absolute observable increment $\left| \Delta v(f^t(x)) \right|$ achieved by some forward iterate of some point in that ball always reaches at least $S_*$. On another note, reformulating Theorem \ref{thm: finite transient time} in terms of its negation, we observe that a point $x_* \in X^v$ is not a $v$-transient center precisely when for every possible threshold $S_* >0$, there is some neighborhood $U$ around $x_*$ whose points $x \in U$ satisfy $T_{S_*}(x) = \infty$. In other words, all orbits starting in this neighborhood $U$ never exceed some threshold in the rate of change of the observable $v$. 

\begin{example}
    Consider the following model system given by
    \begin{align}
    \begin{cases}
        x(t+1) &= x(t) [1-hy(t)] \\
        y(t+1) &= y(t) +h [x(t) - 1] 
    \end{cases} \label{model: example1}
    \end{align}
where $h \in (0,1)$. Let $v(x, y) = x$. We show that the origin $(0,0)$ is a $v$-transient center. Let
$
    f(x, y) = \left( x(1-hy), y +h(x-1)\right).
$
Then $\Delta v(x, y) = v(f(x, y)) - v(x, y) = -h x y$. Our strategy is to apply Theorem \ref{thm: finite transient time}. 
It is easy to show that $f^t(0,y)=(0,y-th)$ so it is clear that $(0, 0) \in X^v$. Set $S_* \in (0, h^2)$ and let $U$ be an arbitrary neighborhood of the origin. Let $\xi = (\varepsilon, 0)$ where $\varepsilon \in\left(0, 1 - \frac{S_*}{h^2}\right)$ is small enough so that $\xi \in U$. For simplicity, we denote by $\{(x_t, y_t)\}_{t \in \mathbb{Z}^+_0}$ the trajectory $\{f^t(\xi)\}_{t \in \mathbb{Z}^+_0}$. Clearly, $|\Delta v(x_0, y_0)| =0 < S_*$. In the following, we show that there exists $t_* \in \mathbb{Z}^+$ such that $|\Delta v(x_{t_*}, y_{t_*})| > S_*$. 

Note that as long as $x_t < 1$, we see from model system \eqref{model: example1} that 
\begin{align*}
y_{t+1} = y_t + h (x_t -1) < y_t.
\end{align*}
Thus, if we define $t_* = \min \{t \in \mathbb{Z}^+~:~x_t \geq 1\}$, then we have $y_t < y_{t-1}< \dots < y_1$ for all $t \in 1 \mathbin{:} t_*$. Specifically, we have $y_{t_*} < y_1 = -h(1-\varepsilon) <0$. Moreover, whenever $y_t < y_1 <0$ for any $t \in \mathbb{Z^+}$,
\begin{align*}
    x_{t+1}  = x_t(1-hy_t) > x_t (1-hy_1) = (1+\alpha) x_t
\end{align*}
where $\alpha = -hy_1 = h^2(1-\varepsilon) \in (0,1)$.  
This implies that
\begin{align}
    x_t \geq x_1(1+\alpha)^{t-1}  = \varepsilon (1+\alpha)^{t-1} \label{ineq: lowerbound1}
\end{align}
for any $t \in \mathbb{Z^+}$ with $y_t < y_{t-1} < \dots < y_1$.
Since $\alpha >0$, the right-hand side of \eqref{ineq: lowerbound1} is increasing in $t$. Hence, there is a finite $t \in \mathbb{Z}^+$ with $x_t \geq 1$. Therefore, $t_*$ is well-defined and finite. At time $t_*$, we obtain
\begin{align}
    |\Delta v(x_{t_*}, y_{t_*})| = h x_{t_*} (-y_{t_*}) \geq h^2(1-\varepsilon) > S_*. \label{ineq: lowerineq: bound2}
\end{align}
We see from \eqref{ineq: lowerineq: bound2} that the $(v, S_*)$-transient time of the initial condition $\xi = (\varepsilon, 0)$ satisfies $T_{S_*}(\xi) \leq t_* < \infty$. Therefore, we conclude from Theorem \ref{thm: finite transient time} that the origin is a $v$-transient center of the model system \eqref{model: example1}.

To illustrate these results, we simulate the discrete-time system \eqref{model: example1} subject to $h=0.1$ and initial condition $\xi =(10^{-3}, 0)$. Figure \ref{fig: example1} shows the trajectories $x_t$ and $y_t$ with the dashed line marking the first time $x_t$ crosses 1, and the time‐series of $|\Delta v(x_t, y_t)|$ with the horizontal threshold $S_* >0$, and the $(v, S_*)$-transient time $T_{S_{*}}(\xi) $ highlighted.

\begin{figure}
\centering
\subfloat[]{%
\resizebox*{7cm}{!}{\includegraphics{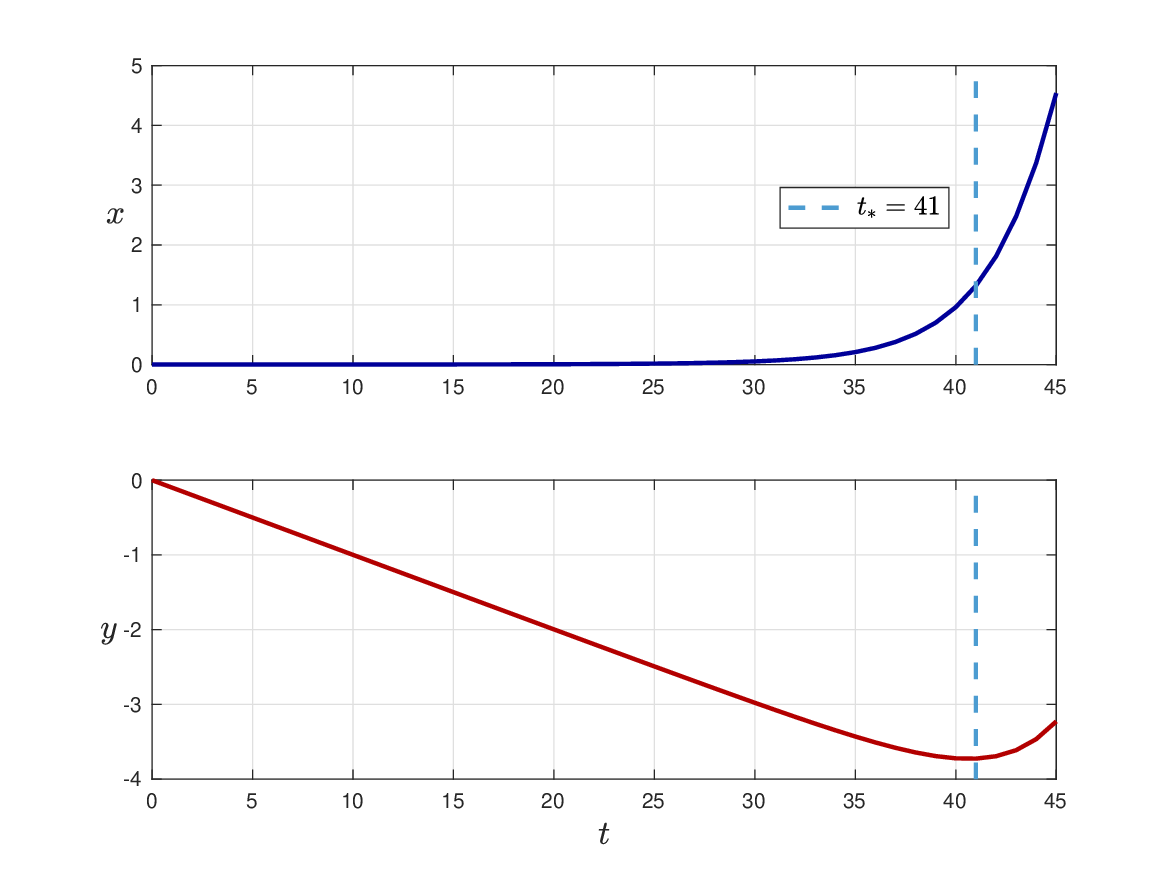}}}\hspace{5pt}
\subfloat[]{%
\resizebox*{7cm}{!}{\includegraphics{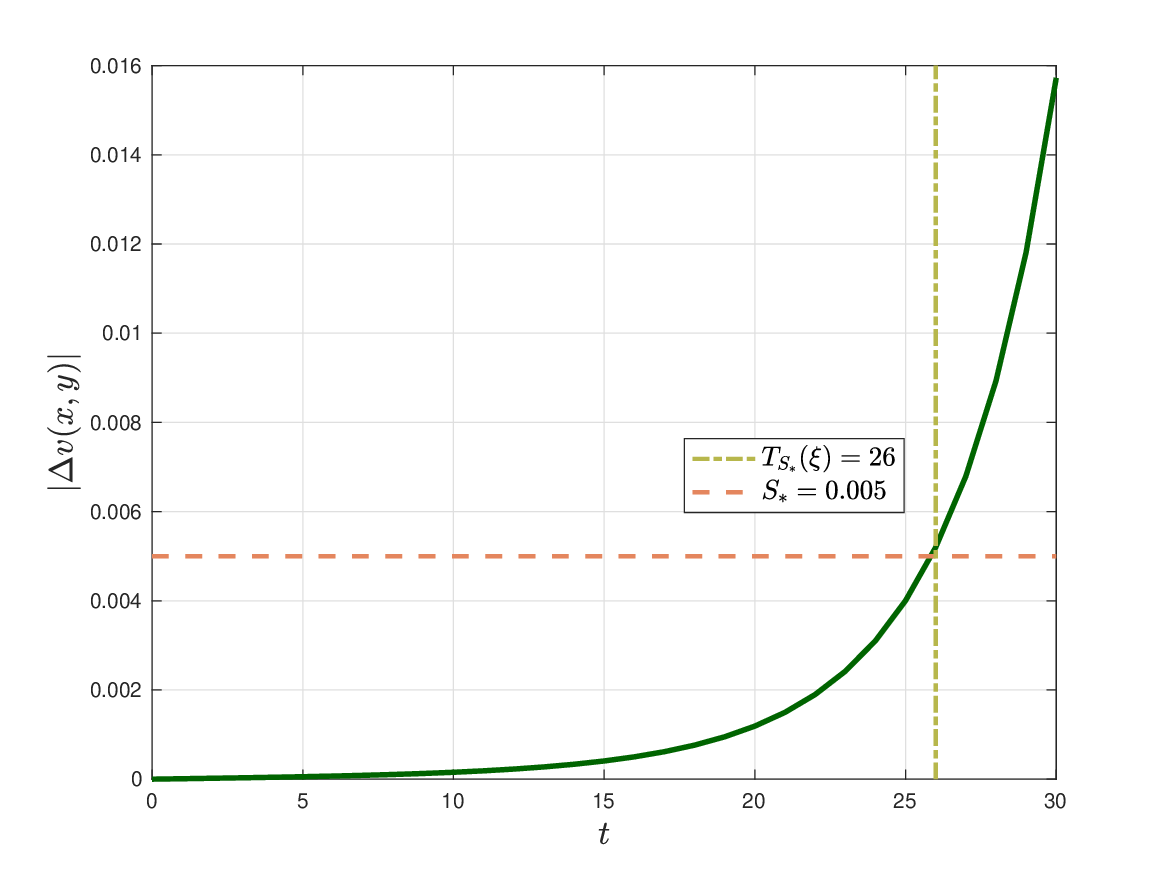}}}
\caption{(a) Trajectories of system \eqref{model: example1} with the dashed vertical line marking $\displaystyle t_*=\min\{t\in \mathbb{Z}^+~:~ x_t \geq 1\}$; and (b) time series of $|\Delta v(x_t, y_t)|$ with the dashed horizontal line at  
  $S_*=\tfrac12\,h^2$ and the dash–dot vertical line at  
  the $(v, S_*)$-transient time $T_{S_*}(\xi)$. } \label{fig: example1}
\end{figure}
\end{example}

\section{Fixed Points as Transient Centers}

In this section, we provide conditions for fixed points to be transient centers. Note that if $x^*$ is a fixed point of model system \eqref{dtades}, then $f(x^*)  = x^*$, and so $\Delta v(x^*) = v(f(x^*)) - v(x^*) =0$ for any $v \in C(\mathbb{R}^n, \mathbb{R})$. Thus, $x^* \in X^v$, and so the fixed points of \eqref{dtades} belong to our candidates for transient centers. Our next result tells us that stable nodes, stable spirals, and centers cannot be transient centers for any choice of the observable $v$.

\begin{theorem}\label{thm: stable}
	A Lyapunov stable fixed point $x^*$ cannot be a $v$-transient center for any observable $v \in C(\mathbb{R}^n, \mathbb{R})$.
\end{theorem}

\begin{proof}
Fix $s>0$.  By the continuity of $v$, there exists $\varepsilon > 0$ such that $|v(x) - v(x^*)| < \frac{s}{2}$ for all $x \in B_\varepsilon(x^*)$. 
By the Lyapunov stability of $x^*$, we can choose $\delta > 0$ such that if $x\in B_\delta(x^*)$ then $f^t(x)\in B_\varepsilon(x^*)$ for all $t \in \mathbb{Z}^+_0$. Thus, for all $t \in \mathbb{Z}^+_0$ and $x\in B_\delta(x^*)$,
	\begin{align*}
			|\Delta v(f^t(x))| &= |v(f^{t+1}(x)) - v(f^t(x))|  \\ &\leq |v(f^{t+1}(x)) - v(x^*)| + |v(x^*) - v(f^t(x))| < \frac{s}{2} + \frac{s}{2} = s.
	\end{align*}
	This implies $|\Delta v(f^t(x))| < s$ for all $t \in \mathbb{Z}^+_0$, and so $ T_s(x) = \infty$. This is true for all $s>0$. From Theorem \ref{thm: finite transient time}, $x^*$ is not a $v$-transient center.
\end{proof}

We now consider conditions under which an unstable fixed point is a $v$–transient center. In the continuous‐time setting, Liu et al. \cite{liu2023framework} make essential use of the natural invertibility of the flow to control both forward and backward trajectories. By contrast, most ecological and biological processes modeled using discrete-time maps are not naturally invertible. Hence, the results we establish here shall only rely on sufficient smoothness conditions on the map $f$ or the observable $v$ with no assumption of invertibility. Let us first consider the linear system given by 
\begin{align}
    x(t+1) = A x(t), \quad A\in\mathbb{R}^{n\times n}. \label{dtades: linear}
\end{align}
The theorems below present simple and easily verifiable sufficient criteria under which the fixed point of \eqref{dtades: linear} at the origin is a $v$-transient center. In fact, the next result is a discrete-time analogue of a continuous-time result for linear systems established in \cite{liu2023framework}.

\begin{theorem}
    Let $E^u$ be the unstable eigenspace\footnote{The unstable eigenspace $E^u$ of a matrix $A$ is defined as the span of all generalized eigenvectors correspoding to eigenvalues $\lambda$ with $|\lambda| >1$.} of $A \in \mathbb{R}^{n \times n}$. For any $v \in C(\mathbb{R}^n, \mathbb{R})$, the origin is $v$-transient center of the linear system \eqref{dtades: linear} provided that there exists $w \in E^u$ such that $\Delta v(w) \neq 0$.
\end{theorem}
\begin{proof}
    Since $E^u$ is an invariant subspace of $A$, the restriction of $A$ to $E^u$, denoted by $A_u$, is a linear map from $E^u$ to $E^u$. The map $A_u$ is also invertible since its eigenvalues are exactly those eigenvalues $\lambda$ of $A$ with $|\lambda| >1$. We also note that the eigenvalues of $A_u^{-1}$ are the reciprocals of the eigenvalues of $A^{u}$, that is, $\lambda^{-1}$ is an eigenvalue of $A_u^{-1}$ if $\lambda$ is an eigenvalue of $A_u$.

    Let $w \in E^u$ such that $\Delta v(w) \neq 0$. Note that $w \neq 0$. Otherwise, if $w=0$, then we get $\Delta v(w) = v(Aw) - v(w) = v(0)-v(0) = 0$ contradicting the assumption that $\Delta v(w) \neq 0$. Let $U$ be an arbitrary neighborhood of the origin. Since $\|A_u^{-t} w\| \to 0$ as $t \to\infty$, there exists a sufficiently large $t_* \in \mathbb{Z}^+$ such that $\xi = A_u^{-t_*} w \in U$. If we set $S_* = \frac{|\Delta v(w)|}{2}$, then the point $\xi$ has the property $T_{S_*} (\xi) \leq t^* < \infty$. By Theorem \ref{thm: finite transient time}, the origin is a $v$-transient center.
\end{proof}

We now replace the pointwise condition $\Delta v(w) \neq 0$ for some $w \in E^u$ with a local condition at the origin: $\nabla v(0) \cdot w \neq 0$, where $w$ is an eigenvector of $A$ corresponding to a real eigenvalue $\lambda$ with $|\lambda|>1$. This local condition ensures that for small initial conditions in the direction of $w$, the absolute observable increment $|\Delta v(x)|$ becomes sufficiently large along the trajectory of the linear system. Additionally, we emphasize that all remaining results and proofs in this section are entirely new contributions that do not presently have continuous-time counterparts.

\begin{theorem}\label{thm: linear_nonorthogonality}
    Let $f (x) = Ax$ where $A\in\mathbb{R}^{n\times n}$ and $v \in C^1(\mathbb{R}^n, \mathbb{R})$. If $A$ has a real eigenvalue $\lambda$ satisfying $|\lambda| > 1$ with corresponding eigenvector $w \in \mathbb{R}^n$ such that $\nabla v(0) \cdot w \neq 0$, then the origin is a $v$-transient center.
\end{theorem}
\begin{proof}
The proof involves finding an $S_*$ with the properties given in Theorem \ref{thm: finite transient time}. Let us denote by $I$ the $n\times n$ identity matrix. We assume without loss of generality that $\lambda >1$ and $\|w\|=1$. Let $U$ be an arbitrary neighborhood of the origin. For each $t\in\mathbb{Z}_0^+$, we have $\Delta v(f^t(x)) = v(A^{t+1} x) - v(A^t x)$. Since $v \in C^1(\mathbb{R}^n, \mathbb{R})$, by the Mean Value Theorem,
    \begin{align*}
        v(A^{t+1} x) - v(A^t x) = \nabla v(\zeta_t) \cdot (A^{t+1} x - A^t x)
    \end{align*}
    for some $\zeta_t \in \mathbb{R}^n$ lying on the line segment joining $A^t x$ to $A^{t+1} x$. Thus,
    \begin{align}
        |\Delta v(f^t (x))| = |\nabla v (\zeta_t) \cdot A^t(A- I)x| . \label{eqn_mvt}
    \end{align}
    Let $\kappa = |\nabla v(0) \cdot w | > 0$. By using the continuity of $\nabla v$, we can find an $r >0$ such that if $y \in B_r(0)$, then
    \begin{align}
        |\nabla v(y) \cdot w - \nabla v(0) \cdot w| < \frac{\kappa}{2}. \label{eqn_continuity}
    \end{align}
Set $x = \varepsilon w$ where $\varepsilon \in(0,r)$ is sufficiently small so that $x \in U$. 
We construct $S_* >0$ such that $T_{S_*} (x) < \infty$. 
Let $t_r = \max\{t\in \mathbb{Z}^+_0~:~\|A^{t} x\|  <  r\}$. Note that $t_r < \infty$ since $A^{t} x=\varepsilon A^{t} w=\varepsilon \lambda^t w$ and $\lambda>1$. We can make $t_r$ as large as we want by making $\varepsilon >0$ sufficiently small. Hence, the point $\zeta_t$ from \eqref{eqn_mvt} lying between the line segment joining $A^t x$ and $A^{t+1}x$ is in $B_r(0)$ for all $t < t_r$. From \eqref{eqn_continuity}, we have 
    \begin{align}
    |\nabla v(\zeta_t)\cdot w - \nabla v(0) \cdot w| < \frac{\kappa}{2} \quad \forall~t \in 0\mathbin{:} t_r-1. \label{eqn_cont2}
    \end{align}
    Hence, we have 
    \begin{align}
        |\nabla v(\zeta_t)\cdot w| \geq |\nabla v(0)\cdot w| - \frac{\kappa}{2} = \frac{\kappa}{2}\quad \forall~t \in 0\mathbin{:} t_r-1.
    \end{align}
    Using equation \eqref{eqn_mvt}, we get for all $t < t_r$ that 
    \begin{align}
        |\Delta v(f^t (x))| = \varepsilon (\lambda-1) \lambda^{t} ~|\nabla v(\zeta_t) \cdot w|  \geq \frac{\varepsilon \kappa}{2} (\lambda -1) \lambda^t. \label{ineq_last}
    \end{align}
    Let $S_* = \frac{r\kappa}{4 \lambda^2}(\lambda -1) >0$ and $t_* = t_r -1 < \infty$. Note that from the definition of the time $t_r$, we have $ \varepsilon \lambda^{t_r + 1} = \|A^{t_r+1} x\|  \geq r$ which implies that $\lambda^{t_*} \geq \frac{r}{\varepsilon \lambda^2}$. Applying \eqref{ineq_last}, we see that
    \begin{align*}
        |\Delta v(f^{t_*} (x))| \geq \frac{\varepsilon \kappa}{2} (\lambda -1) \lambda^{t_*} \geq  2S_* > S_*.
    \end{align*}
    This implies that $T_{S_*} (x) \leq t_* < \infty$ completing the proof. 
\end{proof}

Note that the linear map expands vectors along its unstable directions. If we have two very close level sets of the observable function $v$, say $\{v =c\}$ and $\{v = c + S\}$ where $c \in \mathbb{R}$ and $S>0$ is small, then the condition $\nabla v(0) \cdot w \neq 0$ implies that the line $\{\varepsilon w~:~\varepsilon >0\}$ intersects these two level sets at a nonzero angle. Starting at an initial condition $\varepsilon w$ sufficiently close to the origin, each iteration of the linear system \eqref{dtades: linear} moves the orbit forward along the $w$-direction by a factor of $\lambda >1$. Since $\varepsilon >0$ is small and the gradient $\nabla v$ is nearly constant in a neighborhood of the origin, the increment $|\Delta v|$ remains very small for many steps and so the orbit stays trapped between the two level sets. However, the exponential factor $\lambda^t$ makes the orbit leave this tube of level sets after enough iterations causing $|\Delta v|$ to eventually exceed $S$. We prove an extension of this result for nonlinear systems in the following theorem.

\begin{theorem}\label{thm: nonlinear_nonorthogonality}
Let $x^*$ be a fixed point of \eqref{dtades} where $f\in C^2(\mathbb{R}^n, \mathbb{R}^n)$, and let $v \in C^1(\mathbb{R}^n, \mathbb{R})$. Suppose that $A = Df(x^*)$ has an eigenvalue $\lambda \in \mathbb{R}$ satisfying $|\lambda| >1$ and either
\begin{enumerate}
    \item [(H1)] $|\lambda|^2 > \|A\|$ where $\|A\|$ denotes the spectral norm of $A$, or
    \item [(H2)] $|\lambda|^2 > \rho(A)$ where $\rho(A)$ is the spectral radius of $A$.
\end{enumerate}
If there exists an eigenvector $w \in \mathbb{R}^n$ corresponding to $\lambda$ such that $\nabla v(x^*) \cdot w \neq 0$, then $x^*$ is a $v$-transient center.
\end{theorem}

\begin{proof}
We assume without loss of generality that $x^* =0$ and rewrite \eqref{dtades} as,
\begin{align}
    x(t+1) = A x(t) + g(x(t)), \label{dtades: 2}
\end{align}
where $A = Df(0)$, $g(x) = f(x)-Ax \in C^2(\mathbb{R}^n, \mathbb{R}^n)$, $g(0) =0$, and $Dg(0) =0$. Denote the unique solution $x(t)$ starting at an initial point $\xi \in \mathbb{R}^n$ of \eqref{dtades: 2} by $x_t$ for $t \in \mathbb{Z}^+_0$. From the variation of constants formula (see for instance Theorem 3.17 of \cite{elaydi2005introduction} or Lemma A.1.1 of \cite{cushing1998}), we have $x_0 = \xi$ and
\begin{align}
    x_t = A^t x_0 + \sum_{k=0}^{t-1} A^{t-k-1} g(x_k), \quad t\in\mathbb{Z}^+.
\end{align}

Let $a = \|A\|$, $\kappa = |\nabla v(0) \cdot w|>0$, and assume without loss of generality that $\lambda>1$ and $\|w\|=1$. We consider first the case where (H1) holds so that we have $\lambda^2 > a$. Because $g \in C^2(\mathbb{R}^n, \mathbb{R}^n)$, $g(0)=0$ and $Dg(0) =0$, and $\nabla v$ is continuous with $\nabla v(0) \neq 0$, there are constants $r_0>0$ and $\alpha>0$ such that for all $x\in B_0 := \{x \in \mathbb{R}^n~:~ \|x\| \leq r_0\}$, we have
\begin{align}
   \|g(x)  \| \leq \alpha \|x\|^2  \quad \text{and} \quad |\nabla v(x) \cdot w| \geq \frac{\kappa}{2}.\label{ineq: quadratic}
\end{align}
We also set $\eta:=\sup_{ x\in B_0} \|\nabla v(x)\| < \infty$.
We choose
\begin{align}
    0 < r < \min\left\{r_0,~ \frac{\lambda^2-a}{2\alpha}, ~\frac{\kappa (\lambda-1)(\lambda^2-a)}{8\eta \alpha(\lambda^2 +1)}\right\} \label{ineq: r}.
\end{align}
Let $U$ be an arbitrary neighborhood of the origin and set $\xi = \varepsilon w$ where $\varepsilon \in \Big(0, \frac{r}{2\lambda}\Big)$ is small enough so that $\xi \in U$. Define  
\begin{align}
    t_r = \max \left\{t \in \mathbb{Z}^+_0~:~ \|A^t \xi\| =\varepsilon \lambda^t < \frac{r}{2}\right\} \label{eq: tr}.
\end{align}
Note that $t_r < \infty$ and we can make $t_r$ as large as we want by making $\varepsilon$ sufficiently small. In the following, we construct $S_* >0$ such that $T_{S_*} (\xi) < \infty$. 

Let $e_0 =0$ and $e_t =\sum_{k=0}^{t-1} A^{t-k-1} g(x_k)$ for $t \in \mathbb{Z}^+$. Then $x_t = A^t\xi  + e_t$ for $t \in \mathbb{Z}^+$. We show first that the trajectory satisfies $\|x_t\| \leq 2 \varepsilon \lambda^t < r$ for all $t\in 0: t_r$ by induction. Clearly, we have $\|x_0\| = \varepsilon \|w\|  < 2\varepsilon < \frac{r}{\lambda} < r$. Now, if $\|x_k\| \leq 2 \varepsilon \lambda^k < r$ for all $k \in 0:t-1$, then we obtain from \eqref{ineq: quadratic} that
\begin{multline}
    \|e_t\| \leq \sum_{k=0}^{t-1} a^{t-k-1} \|g(x_k)\| 
    \leq 
    \sum_{k=0}^{t-1} a^{t-k-1} \alpha (2\varepsilon\lambda^k)^2
    \\
    =  4 \alpha \varepsilon^2 \lambda^{2t-2}
    \sum_{k=0}^{t-1} \Big(\frac{a }{\lambda^2}\Big)^{t-k-1}
    \leq 4 \alpha \varepsilon^2 \lambda^{2t-2}\sum_{k=0}^{t-1} \left(\frac{a}{\lambda^2}\right)^k \leq \frac{4\alpha \varepsilon^2}{\lambda^2-a} \lambda^{2t}. \label{ineq: et}
\end{multline}
The last line is because $\lambda^2>a$. Also, we have $\varepsilon \lambda^{t} \leq \varepsilon \lambda^{t_r} <\frac{r}{2}$ for $t \in 0: t_r$, and so $ \frac{4\alpha \varepsilon \lambda^{t}}{\lambda^2-a} \leq \frac{2\alpha  r}{\lambda^2-a} < 1$ by the choice of $r>0$ from \eqref{ineq: r}. Therefore,
\begin{align*}
    \|x_t\| \leq \|A^t\xi\| + \|e_t\| \leq \varepsilon \lambda^t \left(1 + \frac{4\alpha \varepsilon \lambda^{t}}{\lambda^2-a}\right) \leq 2 \varepsilon \lambda^t \leq 2 \varepsilon \lambda^{t_r} < r, \quad \forall~ t\in 0:t_r.
\end{align*}

We are now ready to construct $S_* >0$ so that the $(v, S_*)$-transient time of $\xi$ is finite. Let
\begin{align*}
    L_t = v(A^{t+1} \xi) - v(A^t \xi) \quad \text{and} \quad R_t = [v(x_{t+1}) - v(A^{t+1} \xi)] - [v(x_t) - v(A^t \xi)].
\end{align*} 
We can then express 
\begin{align*}
    \Delta v(x_t) = v(x_{t+1}) - v(x_t) = L_t + R_t.
\end{align*}
By the mean value theorem, there exist points $\psi_t$ on the segment $ v(A^{t+1} \xi)$ to $v(x_{t+1})$ and $\varphi_t$ on the segment $ v(A^{t} \xi)$ to $v(x_{t})$ such that 
\begin{align*}
    v(x_{t+1}) - v(A^{t+1} \xi) = \nabla v(\psi_t) \cdot e_{t+1} \quad \text{and} \quad  v(x_{t}) - v(A^{t} \xi) = \nabla v(\varphi_t) \cdot e_t.
\end{align*}
We thus obtain for all $t < t_r$ that 
\begin{align*}
    |R_t| &\leq \|\nabla v(\psi_t)\| \|e_{t+1}\| + \|\nabla v(\varphi_t)\| \|e_t\|   \leq \frac{4\eta \alpha \varepsilon^2 (\lambda^2 + 1)}{\lambda^2-a} \lambda^{2t} 
\end{align*}
where we used \eqref{ineq: et} in the last inequality. Moreover, similar to what we have done in \eqref{ineq_last} in Theorem \ref{thm: linear_nonorthogonality}, we have 
\begin{align*}
    |L_t| \geq \frac{\varepsilon \kappa}{2} (\lambda -1) \lambda^t~\text{for all}~t < t_r.
\end{align*}
Hence, using the definition of $t_r$ from \eqref{eq: tr}, we apply these inequalities at time $t_* = t_r-1$ to obtain
\begin{align*}
    |L_{t_*}| \geq \frac{\kappa (\lambda-1)}{4\lambda^2} r \quad \text{and} \quad |R_{t_*}| \leq \frac{\eta \alpha (\lambda^2 +1)}{\lambda^2 (\lambda^2 -a)} r^2.
\end{align*}
Set $S_* = \frac{r\kappa(\lambda-1)}{8\lambda^2} >0$. By the choice of $r>0$ from \eqref{ineq: r}, we get

\begin{align*}
    |\Delta v(x_{t_*})|  &\geq |L_{t_*}| - |R_{t_*}| \\ &= \frac{r}{4\lambda^2} \left(\kappa(\lambda-1) - \frac{4\eta \alpha(\lambda^2+1)}{\lambda^2 -a}r\right)  > \frac{r}{4\lambda^2} \left(\frac{\kappa(\lambda-1)}{2}\right) = S_*.
\end{align*}
We conclude that $T_{S_*}(\xi) \leq t_* < \infty$. Therefore, the fixed point $x^*$ is a $v$-transient center of the nonlinear system \eqref{dtades} from Theorem \ref{thm: finite transient time}. 

The proof for the case in which (H2) holds follows a similar line of argument as the proof for the (H1) case. If $\lambda \in \mathbb{R}$ is an eigenvalue of $A$ such that $\lambda >1$ and $\lambda^2 > \rho(A)$, then there is a $\beta >0$ satisfying $\lambda^2 > \rho(A) +\beta$. Since $\beta>0$, there exists a matrix norm $\|\cdot\|_*$ (which is induced by a vector norm on $\mathbb{C}^n$) such that $\|A\|_* \leq \rho(A) + \beta$ (see for instance Lemma 5.6.10 of \cite{horn2012matrix} or Theorem 4.4 of \cite{householder1958approximate}). Hence, we have $|\lambda|^2 > \|A\|_*$ which is similar to condition (H1). We now see that the rest of the proof proceeds similarly to the (H1) case. Using the vector norm that induces the matrix norm $\|\cdot\|_*$, we can replicate the proof for the (H1) case step by step and obtain analogous estimates. Recall that all norms in a finite dimensional space are equivalent.
\end{proof}

The conditions $|\lambda|^2 > \|A\|$ or $|\lambda|^2 > \rho(A)$ are both practical to verify and geometrically meaningful. Computationally, the spectral norm $\|A\| = \sqrt{\rho(A^\top A)}$ and spectral radius $\rho(A)$ can be estimated using standard numerical methods. Geometrically, these conditions ensure that the growth along the unstable direction $w\in \mathbb{R}^n$ sufficiently dominates the system's overall expansion to overcome nonlinear perturbations. It is also worth mentioning several useful special cases and possible extensions of Theorem \ref{thm: nonlinear_nonorthogonality}. First, we note that if $A = Df(x^*)$ and $Aw = \lambda w$ for some $\lambda \in \mathbb{R}$ and $w \in \mathbb{R}^n$, then the condition $\nabla v(x^*) \cdot w \neq 0$ is equivalent to $\nabla \Delta v(x^*) \cdot w \neq 0$. Also, the condition $|\lambda|^2 > \rho(A)$ for some eigenvalue $\lambda \in \mathbb{R}$ is automatically satisfied when $|\lambda| = \rho(A) >1$. Regarding the smoothness condition, we are currently investigating whether the hypothesis $f \in C^2(\mathbb{R}^n, \mathbb{R}^n)$ can be relaxed to $f \in C^1(\mathbb{R}^n, \mathbb{R}^n)$. Note that the current proof relies on a quadratic bound for the nonlinear term $g(x) = f(x) - Df(x^*)x$, which is guaranteed by $C^2$ smoothness. Under $C^1$ smoothness, we only have $g(x) = o(\|x\|)$, which may not suffice to control the error accumulation over the required time interval. Thus, the question remains open. We are also analyzing the case where we have $|\lambda|^2 > \|A\|$ or $|\lambda|^2 > \rho(A)$ for some complex eigenvalue $\lambda$, using a dilation-rotation reduction on the two-dimensional real invariant subspace. Lastly, we have the following corollary which is a special case of Theorem \ref{thm: nonlinear_nonorthogonality}.

\begin{corollary} \label{Corr: Perron-Frobenius}
    Let $A = Df(x^*) \in \mathbb{R}^{n\times n}$ where $n \geq 2$ and $x^* \in \mathbb{R}^n$ is a fixed point of \eqref{dtades} with $f \in C^2(\mathbb{R}^n, \mathbb{R}^n)$. Set $v(x) = p^\top x$ where $p$ is a nonzero nonnegative vector in $\mathbb{R}^n$. If $A$ is nonnegative, irreducible\footnote{A nonnegative matrix $A$ (i.e., all of its entries are nonnegative) is irreducible if and only if $(I + A)^{n-1}$ is a positive matrix (i.e., all of its entries are positive).} and has an eigenvalue $\lambda$ satisfying $|\lambda| > 1$, then $x^*$ is a $v$-transient center.
\end{corollary}

\begin{proof}
    By the Perron–Frobenius Theorem (see for instance Theorem 8.4.4 of \cite{horn2012matrix}), $\rho(A) >0$ is an algebraically simple eigenvalue of $A$ with a corresponding positive eigenvector $w \in \mathbb{R}^n$. The conclusion is now a straightforward consequence of Theorem \ref{thm: nonlinear_nonorthogonality}.
\end{proof}

\begin{example} \label{example: example2}
    Consider the following discrete-time system
    \begin{align}
        x(t+1) = g\left((x(t), y(t)\right), \quad  y(t+1) = h\left(x(t), y(t)\right) \label{model: example2}
    \end{align}
    where $g(x, y) = \frac{ay}{1+x^2}$ and $h(x, y) = \frac{bx}{1+y^2}$. This model was obtained in \cite{elaydi2005introduction}. We assume that $a$ and $b$ are positive values satisfying $ab >1$. Note that the origin $(0,0)$ is a fixed point of \eqref{model: example2}. Also, the partial derivatives are
    \begin{align*}
        g_x = -\frac{2axy}{(1+x^2)^2}, \quad g_y = \frac{a}{1+x^2}, \quad h_x = \frac{b}{1+y^2}, \quad h_y =  -\frac{2bxy}{(1+y^2)^2}.
    \end{align*}
    Hence, the Jacobian at the origin reduces to $A = \begin{bmatrix}
        0 & a \\ b & 0
    \end{bmatrix}$ which is nonnegative, irreducible and has an eigenvalue $\lambda = \sqrt{ab} >1$.
    If we consider the observable function $v(x, y) = (1, 1) \cdot(x, y) = x+y$, then Corollary \ref{Corr: Perron-Frobenius} implies that the origin is a $v$-transient center of model system \eqref{model: example2}. Figure \ref{fig: example2} shows the time series of several trajectories of the system together with the corresponding magnitude of change in the observable $v$.
    

    \begin{figure}
    \centering
    \includegraphics[width=10cm]{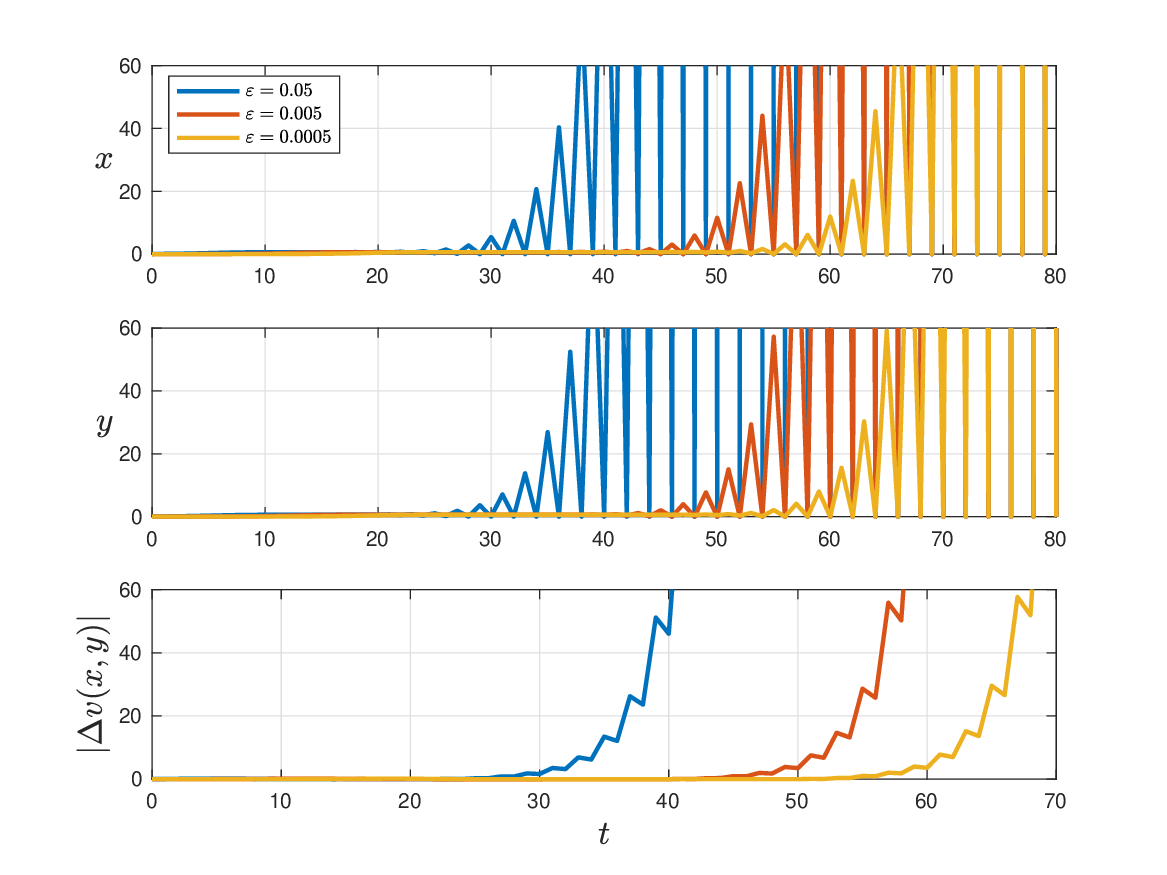}
    \caption{Trajectories of system \eqref{model: example2} for various initial conditions $\xi = \varepsilon w$ with $a =1.5$ and $b =1.3$. The lowest panel shows the magnitude of the observable increment $|\Delta v(x,y)|$ where $v(x,y) = x+y$.}
    \label{fig: example2}
\end{figure}
    
\end{example}


Note that Theorem \ref{thm: nonlinear_nonorthogonality} already breaks down in the situation where $\nabla v(x^*) =0$ because then $\nabla v(x^*) \cdot w=0$ for any eigenvector $w$, and no conclusion can be drawn. Our next result deals with a second-order criterion written in terms of the difference operator $\Delta v(x) = v(f(x)) - v(x)$. This result partially covers the case where we have $\nabla v(x^*) =0$.

\begin{theorem} \label{thm: flatness}
    Let $f \in C^2(\mathbb{R}^n, \mathbb{R}^n)$ and $v \in C^2(\mathbb{R}^n, \mathbb{R})$. If $x^*$ is an unstable fixed point of the nonlinear system \eqref{dtades} such that the gradient $\nabla \Delta v(x^*) = 0$ and the Hessian matrix $H_{\Delta v}(x^*)$ is positive (or negative) definite, then $x^*$ is a $v$-transient center.
\end{theorem}

\begin{proof}
     Since $ \Delta v$ is twice continuously differentiable at $x^*$, we can expand $\Delta v(x)$ using a Taylor series around $x^*$. Since $\Delta v(x^*) =0$ and $\nabla \Delta v(x^*) = 0$, we have
    \begin{align}
    \Delta v(x) =  \frac{1}{2}(x - x^*)^\top H_{\Delta v}(x^*) (x - x^*) + o(\|x - x^*\|^2) \label{eqn: taylor1}
    \end{align}
    where the nonzero Hessian matrix $H_{\Delta v}(x^*)$ is a symmetric matrix whose eigenvalues are real. Additionally, the positive or negative definiteness of $H_{\Delta v}(x^*)$ implies that all of its eigenvalues are either positive or negative, respectively. This assumption also gives us the bound 
    \begin{align*}
        |(x - x^*)^\top H_{\Delta v}(x^*) (x - x^*) | \geq \lambda \| x-x^*\|^2
    \end{align*}
    where $\lambda = \min\{|\lambda_i|~:~ \lambda_i~\text{is an  eigenvalue of $H_{\Delta v}(x^*)$}\}$. Now, by definition of little-o, there is an $r>0$ so small such that 
    \begin{align*}
        \left|o(\|x - x^*\|^2)\right| \leq \frac{\lambda}{4}\|x-x^*\|^2
    \end{align*}
    whenever $x \in \overline{B_r(x^*)} = \{x ~:~\|x-x^*\| \leq r\}$.
    Hence, we have from \eqref{eqn: taylor1} that
	\begin{align}
		|\Delta v(x) | \geq \frac{\lambda}{2} \| x-x^*\|^2, \quad \forall~ x \in \overline{B_r(x^*)}.\label{ineq: bound2}
	\end{align}

    Additionally, since $x^*$ is unstable, there exists some $\rho \in (0, r]$ such that for any neighborhood $U$ of $x^*$, there exists a $y \in U\setminus \{x^*\}$ and a time $t_\rho \in \mathbb{Z}^+$ such that $\|f^{t_\rho}(y) - x^*\| > \rho$. Consider the continuous function $g: [0, 1] \to \mathbb{R}$ defined by
    \begin{align*}
        g(k) = \|f^{t_\rho} \left(x^* + k(y - x^*)\right) - x^*\|.
    \end{align*}
    Note that $g(0) = \|f^{t_\rho}(x^*) - x^*\| =0 < \rho$ and $g(1) = \|f^{t_\rho}(y) - x^*\| > \rho$. Hence, we can apply the Intermediate Value Theorem and conclude that there is some $k^* \in (0,1)$ satisfying $g(k^*) =\rho$. This implies that if we initialize at $x = x^* + k^*( y - x^*) \in U$, then we have $\|f^{t_\rho}(x) -x^*\| = \rho$. 
    We now choose $S_* \in \left(0, \frac{\lambda}{2} \rho^2\right)$ and deduce from \eqref{ineq: bound2} that 
    \begin{align*}
		|\Delta v(f^{t_\rho}(x)) | \geq  \frac{\lambda}{2} \|f^{t_\rho}(x) - x^*\|^2 =  \frac{\lambda}{2} \rho^2 > S_*.
    \end{align*}
    This shows that $T_{S_*}(x) \leq t_\rho <\infty$ for such $x \in U$. Applying Theorem \ref{thm: finite transient time}, we conclude that $x^*$ is a $v$-transient center. 
\end{proof}

\begin{example}
    Consider the following one-dimensional example given by
    \begin{align*}
        x(t+1) = 2 x(t) + [x(t)]^3.
\end{align*}
Let $f(x) = 2 x+ x^3$ and $v(x) = x^2$. Then
\begin{align*}
    \Delta v(x) = v(f(x)) - v(x) = 3x^2 + 4x^4 + x^6.
\end{align*}
Note that $x^* =0$ is an unstable fixed point since $f'(0) = 2 >1$. Also, Theorem \ref{thm: nonlinear_nonorthogonality} does not apply here since $v'(0) = 0$. Nevertheless, we compute that $[\Delta v]'(0) = 0$, and $[\Delta v]''(0) = 6 >0$. Thus, Theorem \ref{thm: flatness} applies and $x^* =0$ is a $v$-transient center.
\end{example}

It is tempting to use the ideas behind the proof of Theorem \ref{thm: flatness} in order to obtain a similar first-order criterion by assuming that $f$ and $v$ are $C^1$, and the gradient $\nabla \Delta v(x^*) \neq 0$. However, this seemingly straightforward adaptation runs into a problem. Unlike in the case of Theorem $\ref{thm: flatness}$ where we obtained a uniform quadratic lower bound in every direction around the fixed point $x^*$, imposing the condition  $\nabla \Delta v(x^*) \neq 0$ can only control the directional growth of $\Delta v$ along the vector $\nabla \Delta v(x^*)$ itself, and not in orthogonal directions. If we Taylor expand $\Delta v$ around $x^*$, then its nonzero linear term satisfies
\begin{align*}
    |\nabla \Delta v(x^*) \cdot (x-x^*)| = \|\nabla \Delta v(x^*)\|\,\|x-x^*\| \,|\cos \theta|
\end{align*}
where $\theta$ is the angle between $\nabla \Delta v(x^*)$ and $\|x-x^*\|$. As $\theta \to \frac{\pi}{2}$, $\cos \theta \to 0$, and thus, no uniform linear bound of the form $|\Delta v(x)| \geq c \|x-x^2\|$, $\forall~ x\in \overline{B_r(x^*)}$ can hold no matter how small $r$ is. The only way to obtain such estimate is to restrict our attention to directions in which we assume (or perhaps know) the directional derivative will not vanish similar to what we did in Theorems \ref{thm: linear_nonorthogonality} and \ref{thm: nonlinear_nonorthogonality}.


\section{Application to Models in Ecology and Epidemiology}

The theoretical framework and criteria for identifying long-lasting and slowly varying transient dynamics developed in the previous sections provide a systematic basis for analyzing a wide class of discrete-time models. In this section, we present concrete examples drawn from ecology and epidemiology to illustrate how the concepts of transient points and transient centers manifest in real-world biological systems. By applying the main results to specific models, we demonstrate how these mathematical tools can be used to identify and predict long transients in populations and disease dynamics. We emphasize through these examples the relevance of long transient dynamics for interpreting and managing ecological and epidemiological systems.

\subsection{Predator-Prey Model}

In 2022, Streipert et al. \cite{streipert2022derivation} presented a new derivation and thorough analysis of a discrete-time predator–prey model, constructed directly from first principles rather than as a discretization of an existing continuous-time model. The model assumes that the prey population exhibits logistic growth in the absence of predators, converging to a carrying capacity, while the predator population requires prey to persist. Most importantly, their approach ensures biologically meaningful dynamics, such as the non-negativity of populations and satisfaction of the axiom of parenthood, that is, no population can arise from zero. The discrete-time model they proposed is given by
\begin{align}
\begin{cases}
x(t+1) = \frac{(1 + r) x(t)}{1 + \frac{r}{K} x(t) + \alpha y(t)} \\[2ex]
y(t+1) = \frac{[1 + \gamma x(t)] y(t)}{1 + d} 
\end{cases} \label{model: predator-prey}
\end{align}
where $x(t)$ and $y(t)$ are the are the prey and predator populations at time $t\in \mathbb{Z}^+_0$, $r>0$ is the intrinsic growth rate of the prey, $K>0$ is the prey population carrying capacity, $\alpha >0$ is the predation rate, $\gamma>0$ is the consumption-energy rate for the predator, and $d>0$ is the predator's natural death rate. 

Consider first the observable function $v_1(x, y) = x$ which depicts the prey population. We first formally prove that the fixed point $E_0 =(0,0)$ is a $v_1$-transient center and therefore cause arbitrarily slow dynamics for arbitrarily long periods of times. We apply Theorem \ref{thm: nonlinear_nonorthogonality}. At $E_0 =(0,0)$, the Jacobian matrix reduces to
\begin{align*}
    A_{E_0} = \begin{bmatrix}
        1+r & 0 \\ 0 & \frac{1}{1+d}
    \end{bmatrix}.
\end{align*}
The eigenvalues are $\lambda_1 = 1 + r > 1$ and $\lambda_2 = \frac{1}{1 + d} < 1$. The corresponding eigenvector for $\lambda_1 = 1 + r$ is $w_1 = (1, 0)$. Since $v_1(x, y) = x$, we have $\nabla v_1(E_0) = (1, 0)$, so $\nabla v_1(E_0) \cdot w_1 = 1 \neq 0$. Thus, all the hypotheses of Theorem \ref{thm: nonlinear_nonorthogonality} are satisfied, and $E_0$ is a $v_1$-transient center. In a similar manner, if we set $v_2(x,y) =y$ and assume that $d < \gamma K$, then the fixed point  $E_K = (K, 0)$ is a $v_2$-transient center. The Jacobian matrix 
\begin{align*}
    A_{E_K} = \begin{bmatrix}
        \frac{1}{1+r} & -\frac{\alpha K}{1+r} \\ 0 & \frac{1+\gamma K}{1+d}
    \end{bmatrix}
\end{align*}
has an eigenvalue $\lambda_1 = \frac{1}{1+r}<1$ and $\lambda_2 = \frac{1 +\gamma K}{1+d} > 1$ since $d < \gamma K$ with corresponding eigenvectors $w_1 = (1, 0)$ and $w_2 = \left(\frac{\alpha K}{(1+r)(\lambda_2 -\lambda_1)},~1\right)$. It is now immediate that $\nabla v_2(E_K)\cdot w_2 = 1 \neq 0$. Applying Theorem \ref{thm: nonlinear_nonorthogonality}, we see that $E_K = (K, 0)$ is a $v_2$-transient center of model system \eqref{model: predator-prey}. It is also a $v_1$-transient center under the same assumption that $d < \gamma K$ since $\nabla v_2(E_K)\cdot w_1 \neq 0$ as well.

\begin{figure}
\centering
\subfloat[]{%
\resizebox*{7cm}{!}{\includegraphics{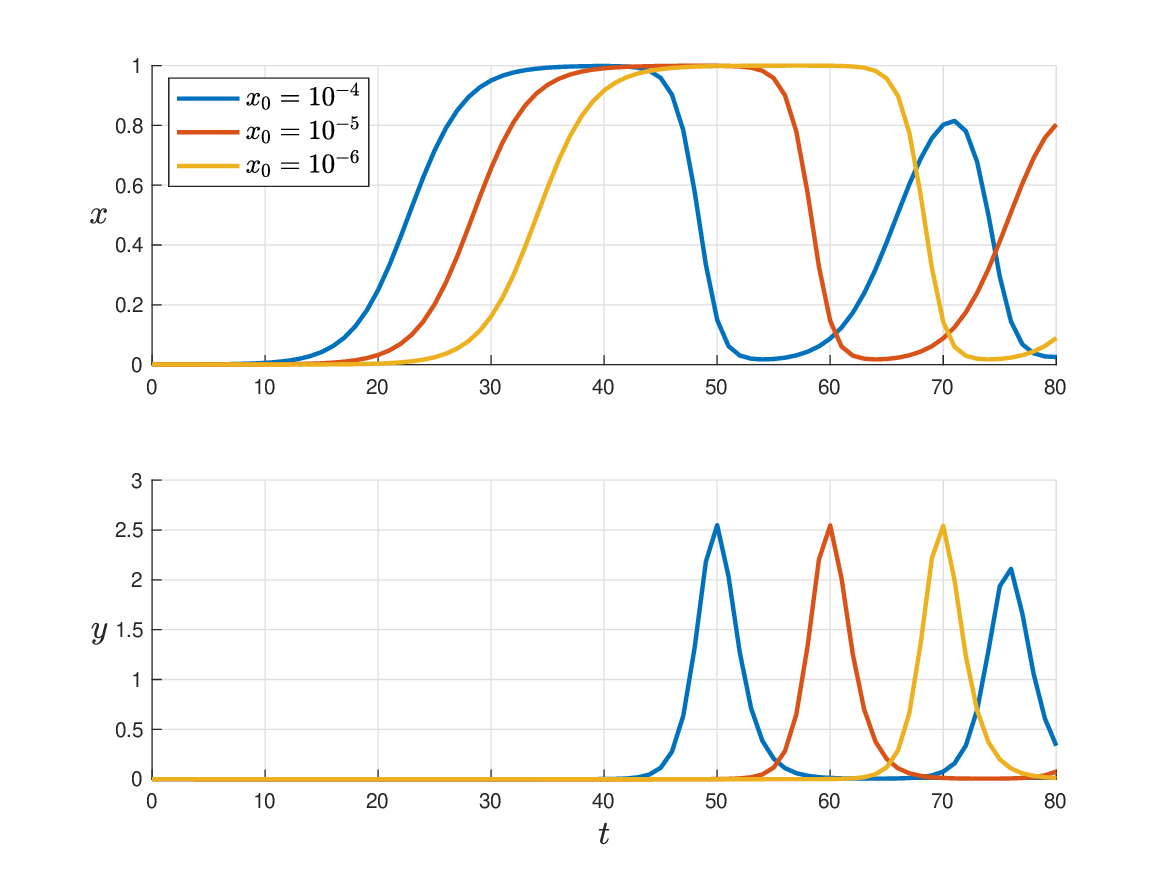}}}\hspace{5pt}
\subfloat[]{%
\resizebox*{7cm}{!}{\includegraphics{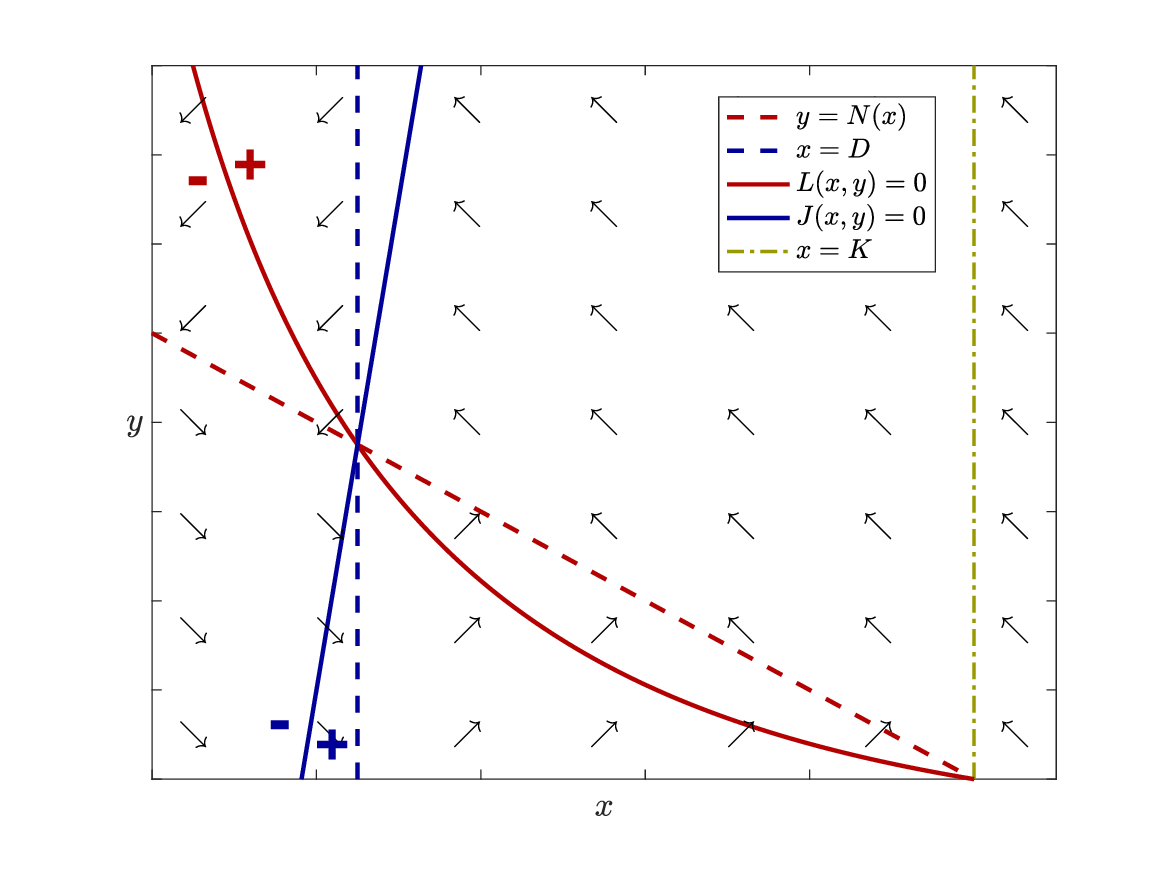}}}
\caption{(a) Sample trajectories of model system \eqref{model: predator-prey} subject to several initial value $x_0$. The other parameters are $r = 0.5$, $K = 1.0$, $\alpha = 1.0$, $\gamma = 4.0$, $d= 1.0$, and $y_0 = 10^{-4}$. (b) The augmented phase portrait for the same predator-prey model. The dashed red and blue lines are the prey ($x$) and predator ($y$) nullclines. The solid red and blue curves are their respective next-iterate root curves. The $'+'$ and $'-'$ symbols indicate the sign of the next-iterate operator in various regions, and the black arrows depict the direction field. The definitions of $D$, $N$, $L$, and $J$ are in the proof of Theorem \ref{thm: pp}.
}
\label{fig: pp_trajectory}
\end{figure}

The trajectories shown in Figure \ref{fig: pp_trajectory}(a) display long-lasting and slowly varying transience about the fixed-point $E_0 =(0,0)$ under the prey population $v_1(x,y) =x$. Furthermore, an examination of the same trajectories shows that the predator population $v_2(x,y) =y$ also remains near zero for an arbitrarily long yet finite period of time. This suggests that $E_0 =(0,0)$ might also be a $v_2$-transient center. The only problem is that neither the hypothesis of Theorem \ref{thm: nonlinear_nonorthogonality} nor that of Theorem \ref{thm: flatness} holds in this case. Nevertheless, the following theorem formalizes this observation. This theorem also shows the existence of transient centers that are not fixed points for model system \eqref{model: predator-prey}.

\begin{theorem} \label{thm: pp} 
    Let $v_2(x,y) =y$. If $d < \gamma K$, then the set $\chi_D = \left\{(x, 0)~:~ 0 \leq x \leq \frac{d}{\gamma} \right\}$ comprises of $v_2$-transient centers. In particular, the fixed point $E_0 =(0,0)$ is a $v_2$-transient center of model \eqref{model: predator-prey}.
\end{theorem}
\begin{proof}
     For simplicity, we denote by $\{(x_t, y_t)\}_{t \in \mathbb{Z}^+_0}$ the trajectory of the predator-prey model \eqref{model: predator-prey}. Note that
    \begin{align}
        \Delta v_2(x,y) =  \frac{y(\gamma x -d)}{1+d} \label{eqn: deltav2}.
    \end{align}
    Using \eqref{eqn: deltav2}, it is easy to check that $\chi_D \subset X^{v_2}$ where $X^{v_2}$ defined in \eqref{X} is the set of all $v_2$-transient center candidates. This is due to $y_t =0$ for all $t \in \mathbb{Z}^+$ provided that $y_0 =0$. We pick an arbitrary element of $\chi_D$ and choose an arbitrary neighborhood $U$ of this member. Our strategy is to apply Theorem \ref{thm: finite transient time}. We construct a transient point $(x_0, y_0) \in U$ whose $(v, s)$-transient time is finite for some $s>0$. Before we do this rigorously, we first outline the geometric intuition of such construction. By using the augmented phase plane approach for discrete planar maps introduced by Streipert and Wolkowicz \cite{streipert2023augmented} in 2023, we perform a phase plane analysis of model system \eqref{model: predator-prey}. The augmented phase portrait is shown in Figure \ref{fig: pp_trajectory}(b). Let
    \begin{align*}
        D = \frac{d}{\gamma}< K \quad \text{and} \quad N(x) = \frac{r}{\alpha}\left(1 - \frac{x}{K}\right).
    \end{align*}
    Note that the lines $y =N(x)$ and $x=D$ are the respective prey ($x$) and predator ($y$) nullclines. Note that $x_t$ is increasing if $y_t < N(x_t)$ and decreasing if $y_t > N(x_t)$. Similarly, $y_t$ is increasing when $x_t >D$ and decreasing when $x_t < D$. We also set 
    \begin{align*}
        L(x_t, y_t) =y_{t+1} - N(x_{t+1}) \quad \text{and} \quad J(x_t, y_t) = x_{t+1}-D.
    \end{align*}
    The functions $L(x,y)$ and $J(x,y)$ are called the next-iterate operators associated with the $y=N(x)$ and $x=D$ nullclines, respectively \cite{streipert2023augmented}. The sign of the operator determines on which side of the associated nullcline the next-iterate lies. If $L(x_t, y_t) >0$ (resp., $L(x_t, y_t) <0$), then $(x_{t+1}, y_{t+1})$ lies above (resp., below) the nullcline $y=N(x)$. If $L(x_t, y_t) =0$, then $(x_{t+1}, y_{t+1})$ lies on $y=N(x)$. The same interpretation holds for the next-iterate operator $J(x, y)$. We add here that the curves defined implicitly by $L(x, y)=0$ and $J(x,y) =0$ are called the next-iterate root-curves associated with the nullclines $y=N(x)$ and $x=D$ \cite{streipert2023augmented}.

    By utilizing the dynamics shown in Figure \ref{fig: pp_trajectory}(b), we establish our design in the following manner. A geometric picture of this set-up is presented in Figure \ref{fig: pp_proof}(a). First, we choose a fixed constant $s>0$ which determines the slow region in such a way that the curve $\Delta v_2(x,y) = s$ intersects the nullcline $y=N(x)$ at exactly two points in the first quadrant. Let us call the $x$-values of the intersection points as $x_*$ and $x_+$ where $x_* < x_+$. Next, we construct an initial point $(x_0, y_0) \in U$ so that the trajectory $(x_t, y_t)$ hits or crosses the curve $\Delta v_2(x,y) = s$ in finite time. For this scenario to happen, the trajectory must enter the region $R = \{(x,y) ~:~ D< x < K,~ 0 <y < N(x)\}$ at some finite time, say $t_{\text{in}} \in \mathbb{Z}^+$. Then, it must exit through the nullcline $y= N(x)$ at some time $t_{\text{out}} > t_{\text{in}}$ where $x_{t_{\text{out}}} \geq x_*$ must hold.

    We shall now give a comprehensive and rigorous proof. Set 
    \begin{align*}
        0 < s < S:= \frac{r\gamma(K-D)^2}{4\alpha K(1+\gamma D)}.
    \end{align*}
    We first show that the curve $\Delta v_2(x, y)=s$ has two intersections with the line $y=N(x)$ in the first quadrant. A quick elementary algebraic manipulation reveals that the roots of the quadratic equation $Ax^2 + Bx +C=0$ where
    \begin{align*}
        A = -\frac{r\gamma}{\alpha K}, \quad B = \frac{r}{\alpha} \left(\gamma + \frac{d}{K}\right), \quad \text{and} \quad  C = - \left[\frac{rd}{\alpha} + s(1+d)\right],
    \end{align*}
    correspond to the $x$ values at the intersections. For a concave down quadratic (since $A <0$) with $B>0$, it suffices that the discriminant $\overline{\Delta} = B^2-4AC$ is positive for it to have two real roots. A straightforward rearrangement tells us that $\overline{\Delta} >0$ is equivalent to $0 < s < S$. Thus, there are two real roots to the quadratic equation, say $x_*$ and $x_{+}$ where $D < x_* < x_+ < K$. We now set
    \begin{align*}
        \eta = \frac{1+r}{1 + \frac{r}{K}x_* + \frac{\alpha}{2} N(x_*)} \quad \text{and} \quad \kappa =  \frac{1+\gamma x_*}{1+d}.
    \end{align*}
    Note that $\eta$ and $\kappa$ are both more than one since 
    \begin{align*}
        1+r = 1 + \frac{r}{K}x_* + \alpha N(x_*) >1 + \frac{r}{K}x_* + \frac{\alpha}{2} N(x_*) 
    \end{align*}
    and $1 + \gamma x_* > 1 + \gamma D = 1+d$. We pick our initial conditions $(x_0, y_0) \in U$ satisfying
    \begin{align*}
        0 < x_0 < D \quad\text{and} \quad 0 < y_0 < \frac{N(x_*)}{2 \kappa^\tau }
    \end{align*}
    where $\tau = \min\{t \in \mathbb{Z}^+ ~:~ D \eta^t \geq x_*\}$. 

\begin{figure}
\centering
\subfloat[]{%
\resizebox*{7cm}{!}{\includegraphics{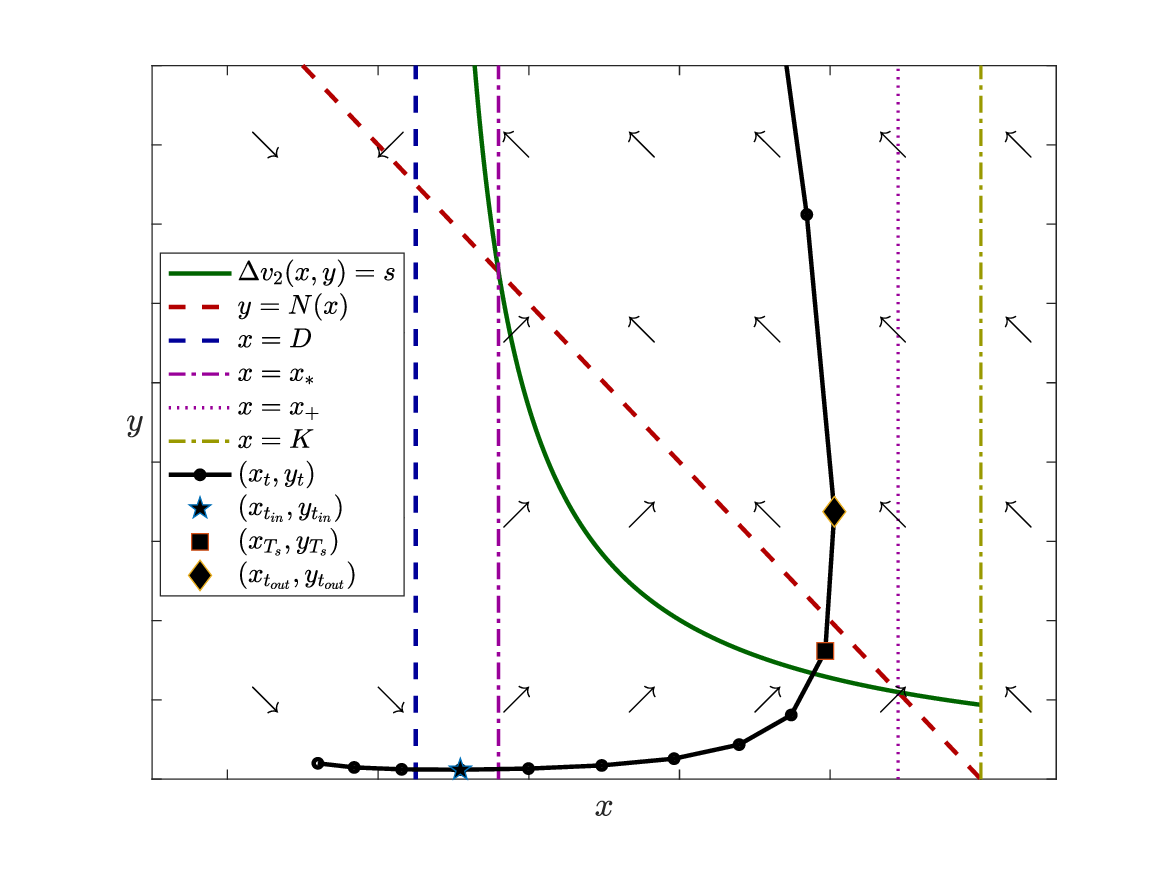}}}\hspace{5pt}
\subfloat[]{%
\resizebox*{7cm}{!}{\includegraphics{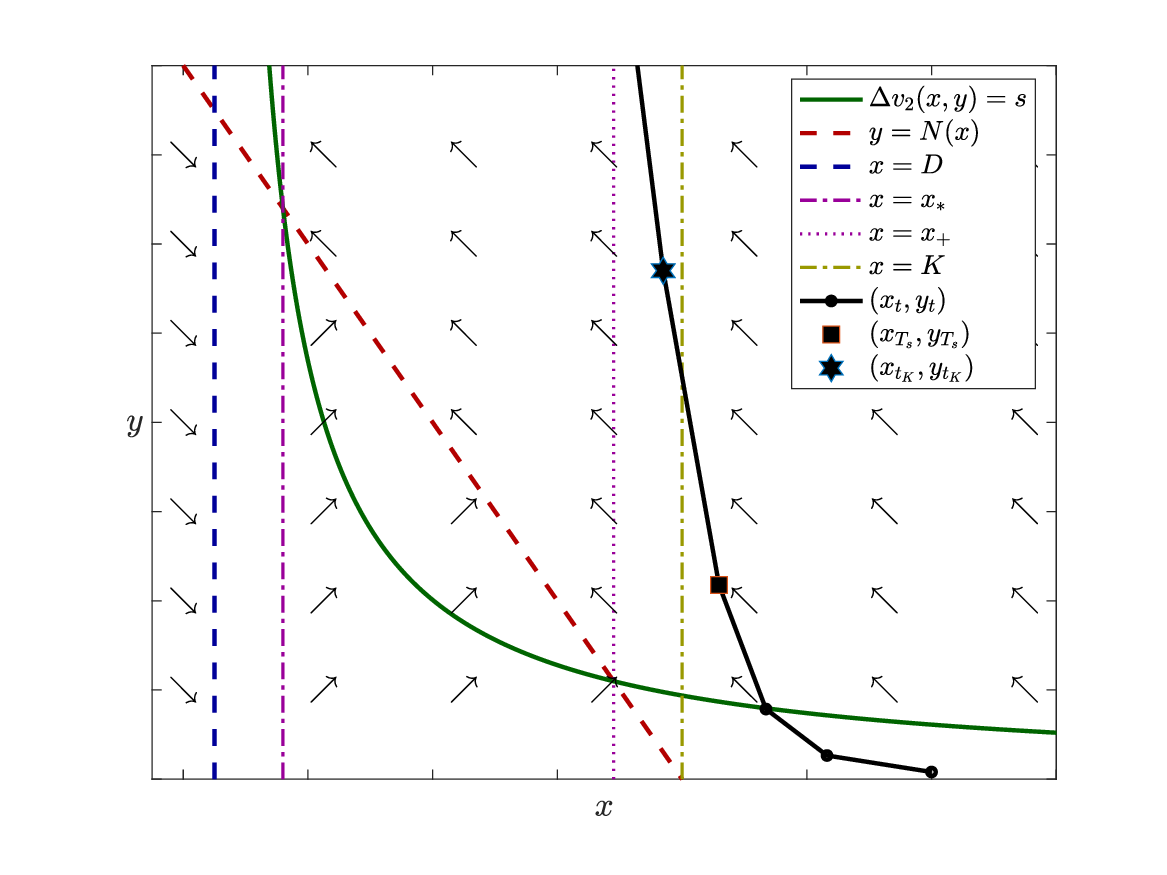}}}
\caption{Illustration of the proof of (a) Theorem \ref{thm: pp} and (b) Theorem \ref{thm: pp1}. } \label{fig: pp_proof}
\end{figure}
    
    In the following, we list some properties of the system's orbit $(x_t, y_t)$ initialized at $(x_0, y_0)$.
    \begin{enumerate}
        \item[(i)] The trajectories satisfy $0 < x_t < K$ and $y_t >0$ for all $t \in \mathbb{Z}^+_0$. 

        We proceed by induction on $t \in \mathbb{Z}^+_0$. Clearly, we have $0 <x_0 < D< K$ and $y_0 >0$. If $0 < x_t < K$ and $y_t >0$ for some $t \in \mathbb{Z}^+_0$, then we clearly have $x_{t+1} >0$ and $y_{t+1} >0$ from $\eqref{model: predator-prey}$. Moreover, we have
        \begin{align*}
            x_{t+1} = \frac{(1 + r)x_t}{ 1+\frac{r}{K} x_{t} + \alpha y_t} \leq \frac{(1+r)x_t}{1 + \frac{r}{K}x_t} < K~\text{if and only if}~x_t < K.
        \end{align*}
    
        \item[(ii)] The orbit enters the region $R = \{(x,y) ~:~ D< x < K,~ 0 <y < N(x)\}$.
        
        It is clear that $y_0 < \frac{N(x_*)}{2 \kappa^\tau} < \frac{N(x_*)}{2} < N(x_*) < N(D)$ where the last inequality follows from the fact that $N(x)$ is strictly decreasing in $x$ on $[0, K]$. This means that $y_0 < N(D)$ as well. While $x_t \leq D$ and $y_t < N(D)$, we see that $y_{t} < N(D) \leq N(x_t)$. This implies that 
        \begin{align*}
            x_{t+1} = \frac{(1+r)x_t}{1 + \frac{r}{K}x_t + \alpha y_t} > \frac{(1+r)x_t}{1 + \frac{r}{K}x_t + \alpha N(x_t)} = x_t
        \end{align*}
        showing that $x_t$ is strictly increasing. Hence, there must exist a finite time $t$ such that $x_t > D$. We let $t_{\text{in}} = \min\{t \in \mathbb{Z}^+~:~ x_t > D\} < \infty$. Moreover, as long as $x_t \leq D$, we have
    \begin{align*}
        y_{t+1} = \frac{(1+\gamma x_t) y_t}{1+d} \leq \left(\frac{1 + \gamma \left(\frac{d}{\gamma}\right)}{1+d}\right) y_t = y_t,
    \end{align*}
    and so $0 <y_{t_{\text{in}}} \leq y_0 < \frac{N(x_*)}{2} < N(D) < N(x_t)$ by induction on $t$. Thus, the orbit has entered the set $R$ at time $t_{\text{in}}$.

    \item[(iii)] The trajectories $x_t$ and $y_t$ are both strictly increasing in $R$.

    Both $x_t$ and $y_t$ are strictly increasing in the region $R$ because 
    \begin{align*}
        x_{t+1} = \frac{(1+r)x_t}{1 + \frac{r}{K}x_t + \alpha y_t} > \frac{(1+r)x_t}{1 + \frac{r}{K}x_t + \alpha N(x_t)} = x_t
    \end{align*}
    and
    \begin{align*}
        y_{t+1} =\frac{1 + \gamma x_t}{1+d} y_t > \frac{1+\gamma D}{1+d} y_t = y_t.
    \end{align*}

    \item[(iv)] Recall that $\tau = \min\{t \in \mathbb{Z}^+ ~:~ D \eta^t \geq x_*\}$. There exists $k\in 0\mathbin{:}\tau$ such that $x_{t_{\text{in}} +k} \geq x_*$.

    We first recall that whenever $x_t < x_*$ and $y_t < \frac{N(x_*)}{2}$,
    \begin{align}
        y_{t+1} = \frac{1 + \gamma x_t}{1+d} y_t \leq \frac{1 + \gamma x_*}{1+d} y_t = \kappa y_t  \label{eqn: kappa}
    \end{align}
    and
    \begin{align}
        x_{t+1} = \frac{(1 + r)x_t}{ 1+\frac{r}{K} x_{t} + \alpha y_t} \geq \frac{1+r}{1 + \frac{r}{K}x_* + \frac{\alpha}{2} N(x_*)} x_t = \eta x_t. \label{eqn: eta}
    \end{align}
    We now proceed to the proof. Assume on the contrary that $x_{t_{\text{in}} + k} < x_*$ for all $k\in 0\mathbin{:}\tau$. A repeated application of \eqref{eqn: kappa} gives us
    \begin{align*}
        y_{t_{\text{in}} + k} \leq \kappa^k y_{t_{\text{in}}} \leq \kappa^k y_0 \leq \kappa^\tau y_0 <\frac{N(x_*)}{2}
    \end{align*}
    for any $k\in 0\mathbin{:}\tau$. This means that we can also apply \eqref{eqn: eta} repeatedly and obtain
    \begin{align*}
        x_{t_{\text{in}} + \tau} \geq \eta^\tau x_{t_{\text{in}}} \geq \eta^\tau D \geq x_*,
    \end{align*}
    which is a contradiction. 
    \end{enumerate}    
    
    Recall that the sequence $\Delta v_2(x_t, y_t)$ is strictly increasing on $R$ since both $x_t$ and $y_t$ are strictly increasing in $R$ as seen in (iii). As $x_t$ and $y_t$ both increase in $R$, the orbit will eventually leave the region $R$ through the line $y= N(x)$ due to (i). Thus, there exists a finite time $t > t_{\text{in}}$ such that $(x_{t}, y_{t})$ satisfies $y_t > N(x_t)$. Let us call this time $t$ as $t_{\text{out}}$. We consider the following cases. 
    \begin{enumerate}
        \item[(a)] $D<x_{t_{\text{out}}} < x_*$; 
        
        Note that this case is not possible since we must have $x_{t_{\text{out}}} \geq x_*$ due to (iv).
        
        \item [(b)] $x_* \leq x_{t_{\text{out}}} \leq x_+$; 
        
        First, we remark that the curve $\Delta v_2(x, y) =  s$ lies below the line $y = N(x)$ on the interval $x_* < x < x_+$ and they intersect at the endpoints. In this particular case, we know that the orbit already crossed the curve $\Delta v_2(x, y) =  s$ by the time the orbit exits the region $R$. This implies that the $(v_2, s)$-transient time $T_{s}(x_0, y_0)$ is finite.
        
        \item [(c)] $x_+<x_{t_{\text{out}}} < K$; 
        
        Note that the curve $\Delta v_2 (x,y) = s$ lies above the line $y = N(x)$ on this interval. At time $t_{\text{out}}$, the orbit now lies in $R_{\text{out}} := \{(x,y)~:~ D < x <K, ~ y > N(x)\}$. In the set $R_{\text{out}}$, we have $x_{t+1} < x_t$ and $y_{t+1} > y_t$. This observation can be easily deduced from the proof of (iii) where we reverse all inequalities. We first show that if $(x_t, y_t) \in R_{\text{out}}$, then $y_{t+1} > N(x_{t+1})$. We write the next-iterate operator into
        \begin{align*}
            L(x_t, y_t) := y_{t+1} - N(x_{t+1}) = \frac{1+\gamma x_t}{1+d}y_t + \frac{r(1+r)x_t}{\alpha K \left(1+ \frac{r}{K}x_t +\alpha y_t \right)} -\frac{r}{\alpha}.
        \end{align*}
        Fixing $D <x_t<K$, we have for all $y_t > N(x_t)$ that
        \begin{align*}
            \frac{\partial}{\partial y_t } L(x_t, y_t) &= \frac{1+\gamma x_t}{1+d} - \frac{r(r+1)x_t}{K \left(1+ \frac{r}{K}x_t +\alpha y_t \right)^2} \\
            & > \frac{1+\gamma D}{1+d} - \frac{r(r+1) x_t}{K \left(1+ \frac{r}{K}x_t +\alpha N(x_t) \right)^2} \\
            & > 1 - \frac{r(r+1)K}{K(1+r)^2} = 1 - \frac{r}{1+r} = \frac{1}{1+r} >0.
        \end{align*}
        So $L(x_t, y_t)$ is increasing in $y_t$ in the set $R_{\text{out}}$. Since 
        \begin{align*}
            L(x_t, N(x_t)) = \frac{r}{\alpha} \left(1 - \frac{x_t}{K}\right) \left(\frac{\gamma x_t-d}{1+d}\right) >0~\text{for all}~ D < x_t < K,
        \end{align*}
        we get $L(x_t, y_t) > L(x_t, N(x_t)) >0$. Hence, $y_{t+1} > N(x_{t+1})$. We have shown that once the orbit exits $R$, it will never go back to $R$ in the succeeding iterates as long as the $x_t$ values satisfy $D < x_t < K$. Now, since $x_{t_{\text{out}}} > x_+$ and $x_t$ is strictly decreasing afterwards, there is a first time $t_+ = \min\{t > t_{\text{out}}~:~ x_t \leq x_+\}$. At this time $t_+$, we know that $x_{t_+} \leq x_+$ and $y_{t_+} > N(x_{t_+})$ as shown above. This means that the orbit already crossed the curve $\Delta v_2(x, y) =s$ before time $t_+$ implying that  $T_s(x_0, y_0) \leq t_{+} < \infty$. 
    \end{enumerate}
      We have shown in (a)-(c) that the orbit crosses the curve $\Delta v_2(x, y) = s$. It is now immediate from Theorem \ref{thm: finite transient time} that any member of $\chi_D$ is a $v_2$-transient center of model system \eqref{model: predator-prey} completing the proof.
\end{proof}

Although our proof above focused on the set $\chi_D$, the very same argument with minor modifications actually applies to every point on the entire nonnegative $x$-axis, that is, every point in the larger set $\chi = \{(x, 0)~:~x \geq 0\}$ is also a $v_2$-transient center. We make this observation precise in the following result, which extends our earlier theorem to the whole nonnegative $x$-axis.

\begin{theorem} \label{thm: pp1}
    Every member of $\chi = \{(x, 0)~:~x \geq 0\}$ is a $v_2$-transient center provided that $d < \gamma K$.
\end{theorem}
\begin{proof}
    First, we give the remark that all members of $\chi$ belong to our set of $v_2$-transient center candidates $X^{v_2}$ defined in \eqref{X}. We just need to consider the cases where $D<x \leq K$ and $x>K$. Throughout this discussion, we retain the definitions of all constructions and notations from the previous theorem, unless otherwise indicated. Note that Theorem \ref{thm: pp} also covers the case where $x \in (D, K]$. Indeed, if we pick a neighborhood $U$ of $(x, 0)$ and initialize at a point $(x_0, y_0)$ in this neighborhood where 
    $D < x_0 < K$ and $0 < y_0 < \min\left\{\frac{N(x_*)}{2 \kappa^\tau }, ~N(x_0)\right\}$, then one is already in the region $R$ and all of the arguments in Theorem \ref{thm: pp} carry through verbatim. We again direct the reader to Figure \ref{fig: pp_proof}(a) for a visual illustration of this straightforward extension. 

    We now show that the point $(\tilde{x}, 0)$ where $\tilde{x} > K$ is a $v_2$-transient center. The arguments below are also illustrated in Figure \ref{fig: pp_proof}(b). Consider the region 
    \begin{align*}
        Q = \left\{(x, y) ~:~ x>K, ~y >0, ~y > N(x), ~\Delta v_2(x,y) < s \right\}.
    \end{align*}
    Let $\varepsilon >0$ be small enough so that $(x_0, y_0) = (\tilde{x} +\varepsilon, \varepsilon) \in Q$. Under the assumption that $d < \gamma K$, it was shown in Proposition 8 of \cite{streipert2022derivation} that if $x_0 >0$ and $y_0>0$, then there exists $T \in \mathbb{Z}^+_0$ such that $x_t < K$ for all $t \geq T$. Hence, starting at $(x_0, y_0) \in Q$, there exists a first time $t_K \in \mathbb{Z}^+$ such that $x_{t_K} < K$. If $\Delta v_2(x_{t_K}, y_{t_K}) \geq s$, then this means that the orbit crossed the curve $\Delta v_2(x, y) = s$ at some time $t \leq t_K < \infty$. Otherwise, if $\Delta v_2(x_{t_K}, y_{t_K}) < s$, then there must exist a finite time $t_* > t_K $ such that $\Delta v_2(x_{t_*}, y_{t_*}) \geq s$ as shown in the proof of Theorem \ref{thm: pp}. In either case, the orbit will surpass the curve $\Delta v_2(x, y) = s$ in finite time. 
\end{proof}


Our findings have a clear biological interpretation. The existence of a $v$-transient center for an observable $v$ corresponding to a population component such as the prey or predator means that, for initial conditions close to the transient center, the population can remain at a constant level for an extended but finite duration before eventually increasing either due to underlying instabilities or as a result of other features in the system’s dynamics. In particular, we have observed that although environmental conditions eventually allow population recovery due to the instability of fixed points, the local dynamics near these repelling fixed points can keep populations trapped in low-density states for prolonged periods. From a management perspective, this means that relying solely on local growth rates around a fixed point may significantly underestimate the true recovery time. Populations near such transient centers might appear permanently collapsed, even when long-term recovery is possible. Moreover, we have demonstrated in model system \eqref{model: predator-prey} that it is not just the isolated fixed-points $E_0$ and $E_K$ that exhibit this transient behavior, but the entire line $\left\{(x, 0)~:~ x\geq0 \right\}$ is made of transient centers for the predator observable. Concretely, any state within this threshold forces the predator population to linger at near-zero levels for a considerable amount of time before eventually rebounding. Ecologically, this provides a rigorous explanation for why predator populations can remain persistently low even when prey densities would otherwise support recovery.

In ecological models, the mechanism behind transient centers provides a rigorous mathematical explanation for temporary population collapse. As seen in model system \eqref{model: predator-prey}, a predator population might seem entirely wiped out for seasons or years only to reappear later, or why a prey population can hang on at vanishingly low densities before suddenly exploding. We shall see in the next subsection that similar behavior occurs in epidemic models. Disease prevalence can remain at minimal levels, nearly unnoticed, before suddenly escalating into a significant outbreak.

\subsection{Epidemic Model with Vaccination}

In this subsection, we consider a discrete-time epidemic model for a measles outbreak obtained from \cite{allen2007introduction}. The model partitions the total population ($N$) into three compartments, namely, susceptible individuals ($S$), infected individuals ($I$), and immune individuals ($R$). The dynamics between these compartments are governed by the following difference equations
\begin{align}
\begin{cases}
S(t+1) = (1-p)S(t) - \alpha S(t)I(t) + b \\[2ex]
I(t+1) =\alpha S(t) I(t) \\[2ex]
R(t+1) =R(t) + I(t) - b + pS(t)
\end{cases} \label{model: epidemic}
\end{align}
where $t$ is the time index in weeks. In the above model, $p\in [0,1)$ is the proportion of susceptible individuals vaccinated per week, $\alpha \in (0,1)$ is the disease transmission rate, and $b>0$ is the constant number of births and deaths. It is assumed that individuals infected with measles recover within a single week, implying that all observed cases each week are newly infected individuals. Moreover, newborns are assumed to enter directly into the susceptible compartment, and every individual will eventually contract the disease so that all deaths occur exclusively among recovered/vaccinated individuals. Summing all equations in system \eqref{model: epidemic}, we deduce that $N(t+1) = N(t)$, and so the population size is a constant value $N$. Since $R(t)$ does not appear explicitly in the first two equations of system \eqref{model: epidemic}, the dynamics are governed by the reduced system
\begin{align}
\begin{cases}
S(t+1) = (1-p)S(t) - \alpha S(t)I(t) + b \\[2ex]
I(t+1) =\alpha S(t) I(t).
\end{cases} \label{model: epidemic_simple}
\end{align}

We remark that $b$ needs to be sufficiently small to ensure that solutions of model system \eqref{model: epidemic} are nonnegative. If there is no vaccination, that is, $p=0$, then there exists only the endemic equilibrium $\overline{S} =1/\alpha$ and $\overline{I} = b$ which is locally asymptotically stable. Meanwhile, when $p>0$, model \eqref{model: epidemic_simple} has two equilibria, the disease-free equilibrium $E_0 = (b/p, 0)$ and the endemic equilibrium $E_* = (1/\alpha,~ \mathcal{R}_0 -1)$ provided that the basic reproduction number $\mathcal{R}_0 := \alpha b/p >1$. The disease-free equilibrium $E_0$ is unstable when $\mathcal{R}_0 > 1$ and the endemic equilibrium $E_*$ is locally asymptotically stable when $ 1<\mathcal{R}_0 < 2/p$ provided that $\alpha b < 2$ \cite{allen2007introduction}. A sample trajectory of model \eqref{model: epidemic_simple} is given in Figure \ref{fig: epidemic_trajectory}. Notice that the number of infected individuals initially declines and remains near zero for an extended period, giving the appearance that the disease has been eliminated. However, we see that it converges to a nonzero equilibrium in the long run. This extended initial phase of apparent disease absence is known as the honeymoon period \cite{mclean1988measles, kharazian2020honeymoon}.

\begin{figure}
    \centering
    \includegraphics[width=10cm]{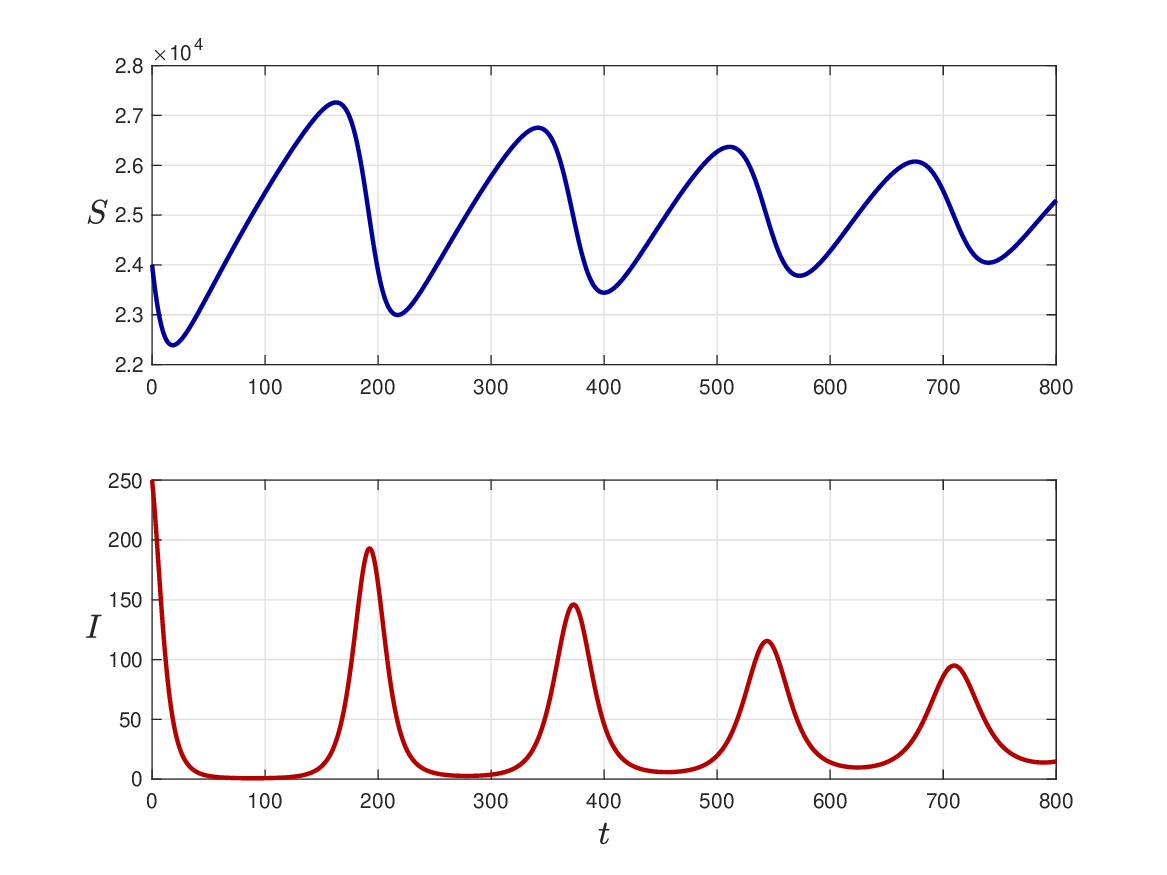}
    \caption{Trajectories of model system \eqref{model: epidemic_simple} subject to $b = 115$, $p = 0.3 \times 10^{-2}$, $\alpha = 0.4 \times 10^{-4}$, $S(0) =2.4 \times 10^4$ and $I(0) =250$. }
    \label{fig: epidemic_trajectory}
\end{figure}

 In the following, we always assume that $p >0$ and $\mathcal{R}_0>1$. We apply the theoretical framework we have developed to characterize long-lasting and slowly varying transient behaviors in model system \eqref{model: epidemic_simple} by considering the observable $v=I$ which is the number of infected individuals. We identify transient points and transient centers within the model's state space and analyze their implications for disease dynamics.

\begin{theorem}
    Let $v(S, I)=I$. Then the fixed point $E_0$ is a $v$-transient center.
\end{theorem}
\begin{proof}
    The Jacobian of the linearized system evaluated at $E_0$ has the eigenvalues $\lambda_1 = \mathcal{R}_0 >1$ and $\lambda_2 = 1-p <1$. The eigenvalue $\lambda_1$ has a corresponding eigenvector $w = (a, 1)$ for some nonzero constant $a$. Hence, $\nabla v(E_0) \cdot w \neq 0$ and it follows from Theorem \ref{thm: nonlinear_nonorthogonality} that $E_0$ is a $v$-transient center. 
\end{proof}

We now tackle the problem of looking for $v$-transient centers that are not fixed points. As usual, we let $\{(S_t, I_t)\}_{t \in \mathbb{Z}^+_0}$ the trajectory of model \eqref{model: epidemic_simple}. We compute that 
\begin{align*}
    \Delta S_t := S_{t+1} - S_t = b-p S_t -\alpha S_t I_t \quad \text{and} \quad 
     \Delta I_t:= I_{t+1}-I_t = \alpha I_t  \left(S_t -\frac{1}{\alpha} \right).
\end{align*}
Since $v = I$, we have $\Delta v(S_t, I_t)= \Delta I_t$. Additionally, if $I_0 =0$, then we have $I_t = 0$ for all $t \in \mathbb{Z}^+$, and so $\Delta v(S_t, I_t)=0$ for any $t \in \mathbb{Z}^+_0$ and any $S_0 \geq 0$. In particular, the set $\Gamma = \left\{(S, 0)~ :~ S\geq0\right\}$ belongs to our set of $v$-transient center candidates. Our next theorem states that all members of $\Gamma$ are indeed $v$-transient centers. This means that even though an outbreak will eventually occur since $\mathcal{R}_0 >1$, a very low number of infective individuals will remain small for a long time.

\begin{theorem} \label{thm: epidemic}
    Consider model system \eqref{model: epidemic_simple} with $v =I$. Then every point in $\Gamma = \left\{(S, 0)~ :~ S\geq0\right\}$ is a $v$-transient center.
\end{theorem}
\begin{proof} [Sketch of the Proof]
    We only give an outline of the proof since the main ideas are similar to the proofs of Theorems \ref{thm: pp} and \ref{thm: pp1}. We see from the augmented phase portrait shown in Figure \ref{fig: epidemic_augmented}(a) that the dynamics of epidemic model \eqref{model: epidemic_simple} share some similarities with the predator-prey model \eqref{model: predator-prey} in the previous subsection. The $S$ nullcline  $I = h(S) := \frac{b-pS}{\alpha S}$ intersects the $I=0$ axis at the disease-free equilibrium $E_0 =(b/p, 0)$, and the nontrivial $I$ nullcline $S = 1/\alpha$ at the endemic equilibrium $E_* = (1/\alpha, ~\mathcal{R}_0-1)$ since $\mathcal{R}_0 = \alpha b/p>1$, that is, $1/\alpha < b/p$. In addition, the nullclines $I= h(S)$ and $S =1/\alpha$ divide the first quadrant into four regions based on the component-wise monotonicity. In Figure \ref{fig: epidemic_augmented}, an arrow pointing to the northeast (resp., southeast) represents $S_{t+1} - S_t >0$ and $I_{t+1}-I_t >0$ (resp., $I_{t+1} - I_t <0$), and an arrow pointing to the northwest (resp., southwest) means that $S_{t+1} - S_t <0$ and $I_{t+1}-I_t >0$ (resp., $I_{t+1} - I_t <0)$. We remark that we added the implicit curve 
    \begin{align*}
        g(S_t, I_t) := (1-p)S_t -\alpha S_t I_t +b =0
    \end{align*}
    since the $S_t$ values can become negative, that is, $S_{t+1} <0$ when $g(S_t, I_t) <0$. This happens when $(S_t, I_t)$ falls above the graph of $g(S, I) =0$.

    We reiterate that the strategy we have done in Theorems \ref{thm: pp} and \ref{thm: pp1} works for the model system \eqref{model: epidemic_simple}. The succeeding claims are best understood by consulting the diagram in Figure \ref{fig: epidemic_augmented}(b).
    Let $(\tilde{S}, 0) \in \Gamma$ and let $U$ be an arbitrary neighborhood of $(\tilde{S}, 0)$. We choose a $k>0$ so that the curve $\Delta v(S, I) = k$ intersects the nullcline $I = h(S)$ in the first quadrant at two points, say $(S_-, I_-)$ and $(S_+, I_+)$. One can verify that any $k$ satisfying 
    \begin{align*}
        0 < k < \frac{b}{\mathcal{R}_0} \left(\sqrt{\mathcal{R}_0}-1\right)^2
    \end{align*}
    will work. We know that the curve $\Delta v(S, I) = k$ has a vertical asymptote at $S = 1/\alpha$ and a horizontal asymptote at $I=0$. Moreover, $\Delta v(S, I) = k$ lies below the nullcline $I = h(S)$ for $S \in (S_-, S_+)$ and above it for $S \notin [S_-, S_+]$. As usual, we need an initial point $(S_0, I_0) \in U$ where the trajectory $(S_t, I_t)$ hits or crosses the curve $\Delta v(S, I) = k$ in finite time. We discuss briefly the construction of such initial point depending on the location of $\tilde{S}$.
    \begin{enumerate}
        \item [(C1)] $0 \leq\tilde{S} \leq 1/\alpha$; ~ The dynamics assure us that if we pick $(S_0, I_0) \in U$ where $0 < S_0 < 1/\alpha$, and $I_0>0$ sufficiently small, the trajectory will enter the region
        \begin{align*}
            P = \left\{(S, I)~:~ S>1/\alpha,~ 0< I < h(S)\right\}.
        \end{align*}
        In the region $P$, both $S_t$ and $I_t$ are increasing. Furthermore, if $S_t \leq b/p$, then
        \begin{align*}
            S_{t+1} < (1-p)S_t + b \leq (1-p)\frac{b}{p} + b = \frac{b}{p} \quad \text{and} \quad I_{t+1} = \alpha S_t I_t \leq \mathcal{R}_0 I_t.
        \end{align*}
        Thus, the trajectory will eventually exit $P$ via the nullcline $I = h(S)$. We fine-tune $I_0 >0$ if necessary so that we stay in the region $P$ until $S_t$ has climbed past $S_-$. We can achieve this since $I_t$ grows by at most $\mathcal{R}_0$ in each iterate. Thus, the trajectory is guaranteed to surpass the curve $\Delta v(S, I) =k$ either before or after exiting the region $P$.

        \item [(C2)] $1/\alpha < \tilde{S} < b/p$; ~This case is also covered in (C1). We initialize at $(S_0, I_0) \in U \cap P$ where $S_0 = \tilde{S} >0$ and $I_0$ is small enough so that  all requirements needed in (C1) are satisfied.

        \item [(C3)] $ \tilde{S} \geq b/p$; ~We initialize at a point $(S_0, I_0) \in U$ where $S_0 > b/p$, $I_0>0$, and $\Delta v(S_0, I_0) < k$. In this region, note that $S_t$ is strictly decreasing and $I_t$ is strictly increasing. Hence, the trajectory will cross the curve $\Delta v(S, I) =k$ irrespective of whether the trajectory crosses the line $S=b/p$ beforehand or afterward.
    \end{enumerate}
\end{proof}

\begin{figure}
\centering
\subfloat[]{%
\resizebox*{7cm}{!}{\includegraphics{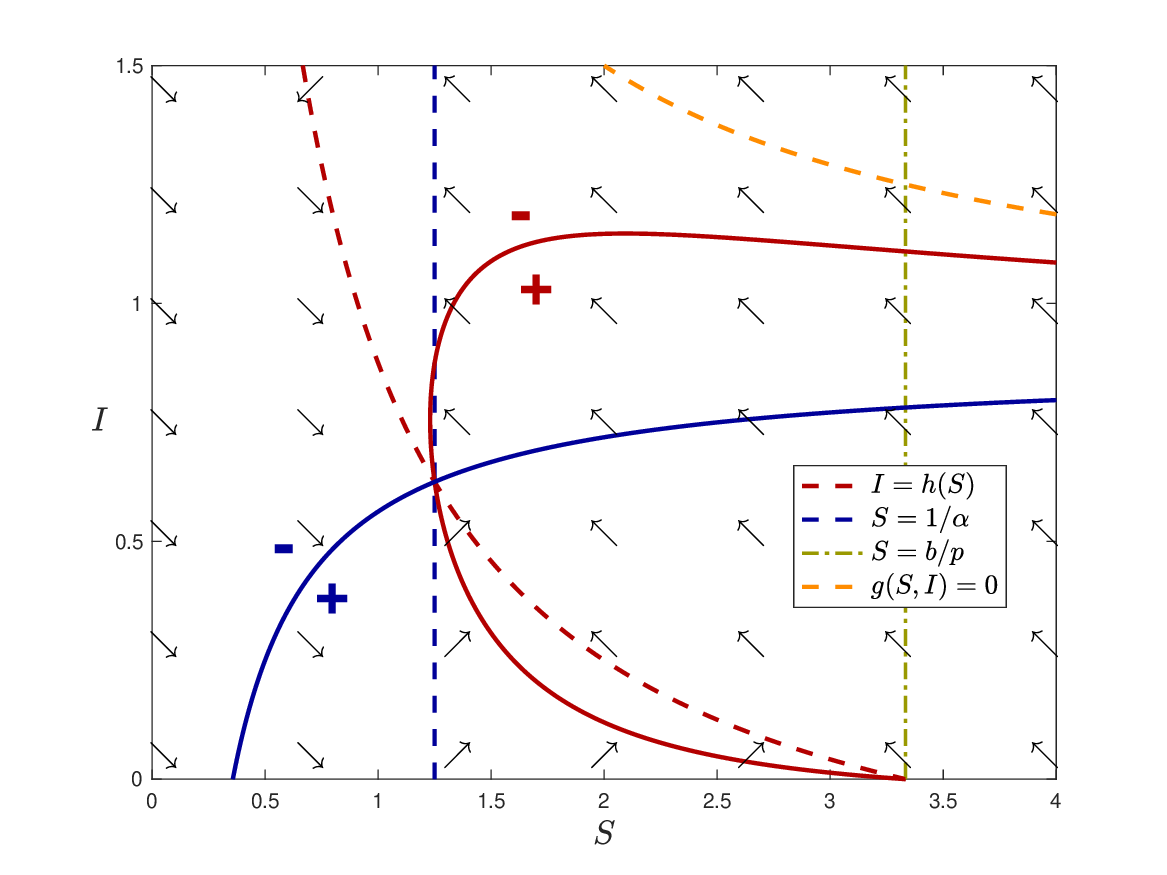}}}\hspace{5pt}
\subfloat[]{%
\resizebox*{7cm}{!}{\includegraphics{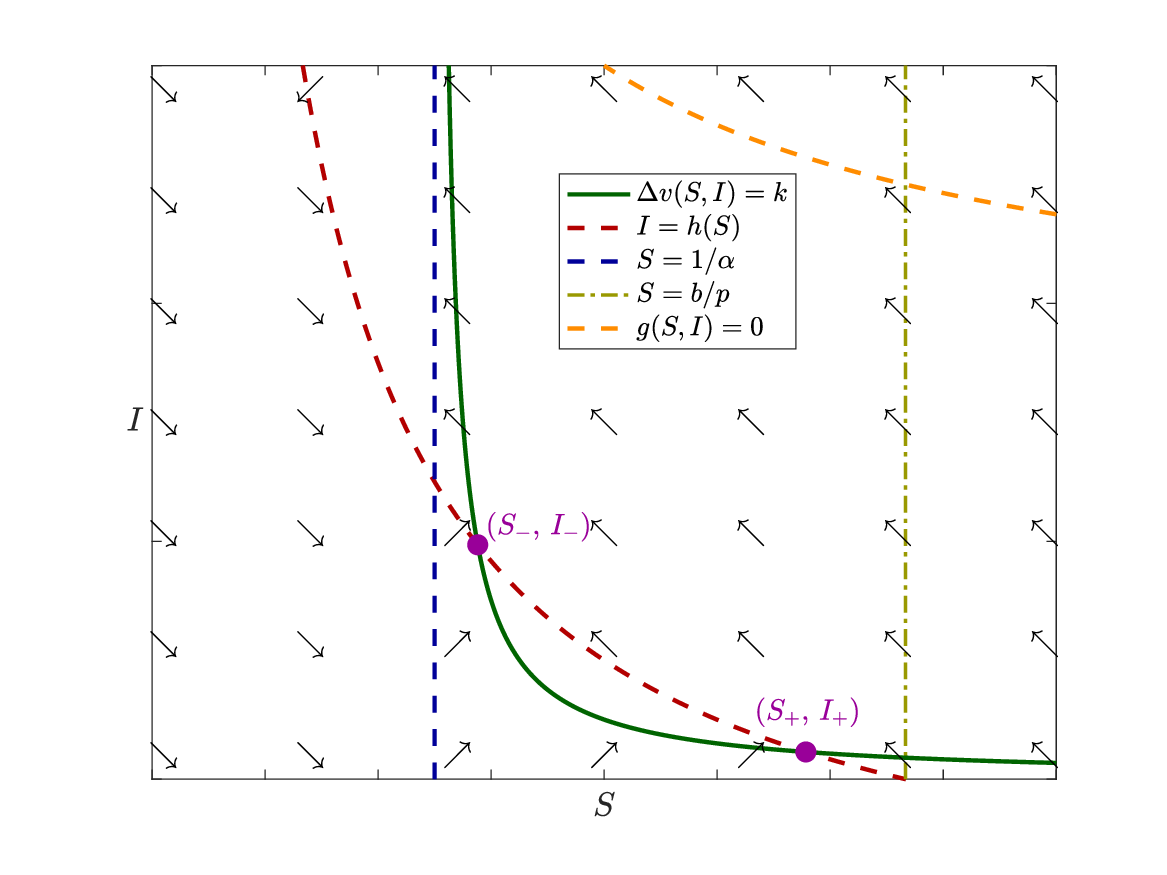}}}
\caption{ The augmented phase portrait based on \cite{streipert2023augmented} for the epidemic model \eqref{model: epidemic_simple} with parameters $p=0.3$, $\alpha =0.8$ and $b=1$. The $S$ and $I$ nullclines are the red and blue dashed curves, respectively. The blue (resp., red) $'+'$ symbol indicates that the next-iterate lies to the right (resp., above) of the blue (resp., red) dashed curve. The blue (resp., red) $'-'$ sign tells us that the next-iterate is located to the left (resp., below) of the blue (resp., red) dashed curve. Points above the orange dashed curve result in negative values which are biologically irrelevant. (b) Illustration of the proof strategy applied in Theorem \ref{thm: epidemic}.
} \label{fig: epidemic_augmented}
\end{figure}

\section{Summary and Future Work}

In this paper, we developed a comprehensive mathematical framework to identify, characterize, and analyze long-lasting and slowly varying transient dynamics in discrete-time model systems. While many previous studies have emphasized stable long-term dynamics like equilibria and periodic cycles, we focused on a specific type of long transient dynamics where a system’s observable remains nearly constant for extended periods before experiencing a sudden change. By extending existing theories from continuous-time models, we introduced precise definitions of transient points and transient centers tailored to discrete-time systems. We also established rigorous conditions and criteria for detecting these structures. In particular, we provided some characterizations of fixed points as transient centers. Lastly, the practical application of the theoretical framework is demonstrated through a detailed analysis of specific examples from predator-prey and epidemic models. The results we obtained provide some important biological insights. They help explain some well-known transient phenomena such as temporary population collapses and the honeymoon period of a disease.

Our current research efforts are concentrated on two main aspects. The first is adapting the concept of reachability which is an important property that makes the transient dynamics generated by transient centers attainable from other points in the state space to discrete-time dynamical systems. In the original formulation of Liu and Magpantay \cite{liu2022quantification}, reachable transient points and reachable transient centers for ODE systems are defined using the backward-time trajectories of the ODE. In this way, if a point in the state space arbitrarily close to a $v$-transient center is a $(v,s, T)$-transient point when you run time forwards and when you run time backwards, then this would enable you to construct an entrance point outside any neighborhood of the given $v$-transient center. The formal, rigorous definition of a reachable $(v,s,T)$-transient point and of a reachable $v$-transient center for ODE systems can be found in \cite{liu2022quantification, liu2023framework}. Several mathematical properties of these points are also established in \cite{liu2023framework}. Of course, unlike invertible flows generated by ODEs, we know that discrete-time maps may not admit a global inverse. Thus, a direct translation to discrete-time systems of this reverse flow requirement is impossible. We are currently developing and investigating a new and purely forward-time definition of reachability that captures the same spirit of attainable transient centers. We emphasize that establishing reachability in discrete-time systems is not merely a formal extension. It ensures that realistic trajectories starting from generic initializations can enter and stay in slow regions for a long duration of time.

The second primary aspect we intend to explore in future research involves transient centers that are not fixed points. Even within the continuous-time setting, the theory on non-equilibrium transient centers remains underdeveloped. Currently, our primary tool for characterizing such transient centers is given in Theorem \ref{thm: finite transient time}. Despite the usefulness of this result in proving various theoretical criteria, it exhibits several drawbacks when applied to concrete examples as demonstrated in the last section. Specifically, verifying the finite transient time condition often requires intricate and careful constructions within the system dynamics which makes the application of this result challenging. 
This motivates a parallel line of work to improve and develop simple, easily verifiable criteria for transient centers, both equilibrium and non-equilibrium, so that identification does not rely on delicate constructions and is accessible to applications.

\section*{Disclosure Statement}
The authors report there are no competing interests to declare.

\section*{Funding}
This work was supported by the Natural Sciences and Engineering Research Council of Canada (NSERC) Discovery Grant Program.

\bibliographystyle{tfs}
\bibliography{references}

\end{document}